\documentclass[11pt,a4paper]{amsart}

	\usepackage{amsfonts,amsmath,amssymb,amsthm}
	\usepackage{times}
	\usepackage{color}
	\usepackage{amscd}

 	\usepackage{multirow,multicol}
	\usepackage[pdftex]{graphicx}					% wichtig fuer includegraphics[...]{Bild}
	\usepackage{calc}
	\usepackage{xspace}
	\usepackage{eucal}
	\usepackage{dsfont}\DeclareSymbolFont{rsfs}{U}{rsfs}{m}{n}
	\DeclareSymbolFontAlphabet{\mathrsfs}{rsfs}
	\usepackage{color}
	
  \newtheoremstyle{style}   
		{}												%Space above    
		{}      			  	        %Space below    
	  {\itshape}                %Body font: original {\normalfont}; {\itshape} fuer kursiven Text in der Satz-Umgebung    
	 	{}                        %Indent amount (empty = no indent,%\parindent = paraindent); 'indent'=einruecken    
		{\normalfont\bfseries}	 	%Thm head font original 
		{\normalfont\bfseries .} {.5em}		%Abstand nach Kopf: {.5em} fuer kleinen horizontalen Abstand; {\newline} fuer Zeilenumbruch
		{}
	\theoremstyle{style}
	\newtheorem{theorem}{Theorem}[section]
	\newtheorem{lemma}[theorem]{Lemma}%[section]
	\newtheorem{corollary}[theorem]{Corollary}%[section]
	\newtheorem{proposition}[theorem]{Proposition}%[section]
	\theoremstyle{remark}
	\newtheorem*{remark}{\textbf{Remark}}
	\theoremstyle{definition}

	\numberwithin{equation}{section}
	
	\allowdisplaybreaks[4]							%erlaubt Seitenwechsel zwischen mehrzeiligen Gleichungen
  \title{On a family of integral operators of Hankel type}
	\date{\today}

  \author[C. Uebersohn]{Christoph Uebersohn}

	\address{\begin{align*}
						&\text{C.~Uebersohn,} \\ &\text{FB 08 - Institut f\"{u}r Mathematik, Johannes Gutenberg-Universit\"{a}t Mainz,} \\
						&\text{Staudinger Weg 9, D-55099 Mainz, Germany}
					\end{align*}}
					\email{uebersoc@mathematik.uni-mainz.de}

\begin{document}
	\begin{abstract}
		In this paper we perform an explicit diagonalization of Hankel integral operators $ K^{(0)}, K^{(1)}, K^{(2)}, ... $ It turns out that each of these 		 operators has a simple purely absolutely continuous spectrum filling in the interval $ [-1,1] $. This generalizes a result of Kostrykin and Makarov 			(2008).
	\end{abstract}
	\maketitle

\section{Introduction and the main theorem}
	We denote by $ \mathbb{S}^{1} = \{ z \in \mathbb{C} : |z| = 1 \} $ and $ H^{2} $ the unit circle and the Hardy space on the circle, respectively. Here 	 and for the rest of the paper we use the following representation of a Hankel operator on $ H^{2} $ with symbol $ \varphi \in 														L^{\infty}(\mathbb{S}^{1}) $:
	\begin{equation*}
		S_{\varphi} = \left. PJM_{\varphi} \right|_{H^{2}}
	\end{equation*} 
	where $ P : L^{2}(\mathbb{S}^{1}) \rightarrow L^{2}(\mathbb{S}^{1}) $ is the orthogonal projection onto the closed subspace $ H^{2} $ of $ 							L^{2}(\mathbb{S}^{1}) $, $ J : L^{2}(\mathbb{S}^{1}) \rightarrow L^{2}(\mathbb{S}^{1}) $ denotes the flip operator, $ (Jf)(z) = f(\bar{z}) $, and $ 		M_{\varphi} : L^{2}(\mathbb{S}^{1}) \rightarrow L^{2}(\mathbb{S}^{1}) $ is the multiplication operator defined by $ M_{\varphi}f = \varphi \cdot f $.
			
	In his 1953 paper \cite{Krein} Mark Krein considered two bounded integral operators $ A_{j} $, $ j=0,1 $, on $ L^{2}(0,\infty) $ with kernels
	\begin{equation*}
		a_{0}(x,y) = \begin{cases}
  								 \sinh(x) \mathrm{e}^{-y}  & \text{if } x \leq y \\
   								 \sinh(y) \mathrm{e}^{-x} & \text{if } x \geq y
 								 \end{cases}
 		~ \text{and} ~
		a_{1}(x,y) = \begin{cases}
  								 \cosh(x) \mathrm{e}^{-y}  & \text{if } x \leq y \\
   								 \cosh(y) \mathrm{e}^{-x} & \text{if } x \geq y
 								 \end{cases}.
	\end{equation*}
	The operators $ A_{j} $, $ j=0,1 $, are the resolvents of the one-dimensional Dirichlet ($ j = 0 $) and Neumann ($ j = 1 $) Laplacians 
	$ - \mathrm{d}^{2} / \mathrm{d}x^{2} $ at the spectral point $ -1 $, respectively. Krein showed that the difference of the spectral projections 				for the operators $ A_{0} $ and $ A_{1} $ is given by
	\begin{equation*}
		K_{\mu} = E_{(-\infty,\mu)}(A_{1}) - E_{(-\infty,\mu)}(A_{0}), \quad 0 < \mu < 1,
	\end{equation*}
	where the kernel function of the integral operator $ K_{\mu} $ is of the form $ k_{\mu}(x+y) $.  More precisely,
	\begin{equation*}
		k_{\mu}(x) = \frac{2}{\pi} ~ \frac{\sin \big( \sqrt{\lambda(\mu)} x \big) }{x} \quad \text{with } \lambda(\mu) = \frac{1}{\mu} - 1.
	\end{equation*}
	In \cite{Kostrykin} Vadim Kostrykin and Konstantin A.\, Makarov deduced from this that $ K_{\mu} $ is unitarily equivalent to $ K_{1/2} $ with kernel 	$ k_{1/2}(x+y) = \frac{2}{\pi} ~ \frac{\sin(x+y)}{x+y} $, and $ K_{1/2} $ is unitarily equivalent to the Hankel operator $ S_{z \phi} $ on $ H^{2} $ 		where
	\begin{equation*}
		\phi : \mathbb{S}^{1} \rightarrow \mathbb{R}, \quad
		\phi(z) = 2 \cdot \mathds{1}_{\overline{\mathbb{H}_{r}} \cap \mathbb{S}^{1}}(z),
	\end{equation*}  
	and $ \mathbb{H}_{r} = \{ z \in \mathbb{C} : \operatorname{Re}(z) > 0 \} $ denotes the right half plane. By $ \mathds{1}_{\overline{\mathbb{H}_{r}} 		\cap \mathbb{S}^{1}} $ we denote the characteristic function of $ \overline{\mathbb{H}_{r}} \cap \mathbb{S}^{1} = \{ z \in 	 \mathbb{S}^{1} : 					\operatorname{Re}(z) \geq 0 \} $.

	In the present paper we consider a "natural" generalization of this operator $ S_{z \phi} $, namely the family of operators $ S_{z \phi}, S_{z^{2} 			\phi}, S_{z^{3} \phi},... $ As we will show in Theorem \ref{Definition der Operatoren K und Beweis einiger Eigenschaften}, these Hankel operators on 	 	 $ H^{2} $ are unitarily equivalent to the Hankel integral operators $ K^{(0)},K^{(1)},$ $ K^{(2)}, ... $ on $ L^{2}(0,\infty) $ with kernels $ 					(0,\infty) \times (0,\infty) \rightarrow \mathbb{R} $, $ (x,y) \mapsto k^{(\ell)}(x+y) $. These kernels are determined by the functions
	\begin{align*}
		&k^{(0)}(x) := \frac{2}{\pi} ~ \frac{\sin(x)}{x}, \\
		&k^{(\ell)}(x) := \frac{1}{\pi} ~ \frac{1}{2^{\ell-1}} \sum\limits_{n=0}^{2 \ell} (-1)^{n} ~\binom{2 \ell}{n} 																														~\frac{\mathrm{d}^{n}}{\mathrm{d}x^{n}} \bigg(\, \big(\, \mathrm{e}^{- \left| w \right|}~p_{\ell}(|w|) \big)\, \ast \Big(\, 														\frac{\mathrm{sin}(w)}{w} \Big)\, \bigg)\,(x),
	\end{align*}
	defined even on the whole of $ \mathbb{R} $, where $ \sin(0)/0 := 1 $,
	\begin{equation*}
	 	p_{\ell}(w) = \sum\limits_{j=0}^{\ell-1} \frac{1}{2^{j}} \binom{\ell+j-1}{\ell-1} \frac{1}{(\ell-j-1)!} w^{\ell-j-1}, \quad  w \in \mathbb{R},
	\end{equation*}
	 and $ \ell \in \mathbb{N} := \{ 1,2,3,... \} $. Here and for the rest of the paper we denote by $ \binom{N}{n} = \frac{N!}{n! (N-n)!} $ the binomial 	coefficient "$ N $ choose $ n $" if $ n $ and $ N $ are in $ \mathbb{N}_{0} := \mathbb{N} \cup \{ 0 \} $ such that $ n \leq N $. Furthermore, if $ f 		\in L^{p}(\mathbb{R}) $ for some $ p \in \left[ 1,\infty \right] $ and $ g \in L^{1}(\mathbb{R}) $, we denote by $ f \ast g $ the convolution of $ f $ 	 and $ g $,
	\begin{equation*}
		(f \ast g)(y) := \big(\, f(w) \ast g(w) \big)\, (y) := \int_{-\infty}^{\infty} f(y-x) g(x) \mathrm{d}x, \quad y \in \mathbb{R}.
	\end{equation*}
		
	In Theorem \ref{Diagonalisierung} we will give a full description of the spectrum of each $ K^{(0)},$ $ K^{(1)}$, $ K^{(2)}, ... $	
	\begin{theorem} \label{Diagonalisierung}
		Let $ p \leq 1/2 $ be a real number. Define the functions $ \rho_{p} : (0,\infty) \rightarrow (0,\infty) $ and $ h : (0,\infty) \rightarrow 
		(0, \infty) $ by
		\begin{equation*}
			\rho_{p}(\lambda) = \frac{1}{2 \pi^{2}}~\mathrm{sinh}\left( 2 \pi \sqrt{\lambda} \right) \left| \Gamma \left( \frac{1}{2}-p- \mathrm{i} 								\sqrt{\lambda} \right) \right|^{2}
			\quad \text{and} \quad
			h(\lambda) = \frac{\pi}{\mathrm{cosh}\left( \pi \sqrt{\lambda} \right)},
		\end{equation*}
		respectively, where	$ \Gamma $ denotes the Gamma function. For all $ m \in \mathbb{N}_{0} $ the following results hold true:
		\begin{enumerate}
			\item The operator $ K^{(2m)} $ is unitarily equivalent to 
					\begin{equation*}
						 M_{\frac{(-1)^{m}}{\pi} h} \oplus M_{\frac{(-1)^{m+1}}{\pi} h}
						 \text{   on   }  L^{2}\left( (0,\infty); \rho_{\frac{1}{2}-m}(\lambda) \mathrm{d}\lambda \right) \oplus L^{2}\left( (0,\infty); 											 \rho_{-\frac{1}{2}-m}(\lambda) \mathrm{d}\lambda \right) 
					\end{equation*}
					and has a simple purely absolutely continuous spectrum filling in the \linebreak interval $ \left[ -1,1 \right] $.
			\item The operator $ K^{(2m+1)} $ is unitarily equivalent to 
					\begin{equation*}
						M_{\frac{(-1)^{m+1}}{\pi} h} \oplus M_{\frac{(-1)^{m}}{\pi} h}
						\text{   on   } L^{2}\left( (0,\infty); \rho_{-\frac{1}{2}-m}(\lambda) \mathrm{d}\lambda \right) \oplus L^{2}\left( (0,\infty); 												\rho_{-\frac{1}{2}-m}(\lambda) \mathrm{d}\lambda \right)
					\end{equation*}
					and has a simple purely absolutely continuous spectrum filling in the \linebreak interval $ \left[ -1,1 \right] $.
		\end{enumerate}
	\end{theorem}
	
	Furthermore, we will prove some properties of the functions $ k^{(1)},$ $k^{(2)}, k^{(3)},... $ such as regularity and integrability, and we will 			study the behaviour of $ k^{(\ell)}(x) $ as $ |x| \rightarrow \infty $, see Theorem \ref{Gesammelte Eigenschaften der Funktionen k} and Corollary 			\ref{Ergaenzung der gesammelten Eigenschaften der Funktionen k}. In Theorem \ref{Satz ueber die explizite Darstellung der Funktionen k} we present an 	"explicit" formula for $ k^{(\ell)}(x) $ if $ \ell \in \mathbb{N} $ and $ x > 0 $, i.\,e., a formula which neither contains integrals nor 							convolutions. From this representation of the functions $ k^{(1)},k^{(2)},k^{(3)},... $ it is immediate that each function $ k^{(\ell)} $ is real 			analytic on $ (0,\infty) $, see Corollary \ref{Die Funktionen k sind reell analytisch fuer positive x}.
	
	We know only a few examples of Hankel operators which can be explicitly diagonalized, see Yafaev \cite{Yafaev}. From this point of view, it would be 		interesting to consider $ K^{(0)}, K^{(1)}, K^{(2)}, ... $ even if we had not known that the operators $ K^{(1)}, K^{(2)}, ... $ can be regarded as a 	"natural" generalization of the Hankel operator of Krein's example. Unfortunately, Yafaev's commutator method for the diagonalization of Hankel 				operators which is described in \cite{Yafaev} is not applicable to our problem.

\section{Proof of the main theorem \label{Darauf wird in Abschnitt 2 verwiesen}}

	On the basis of Theorem \ref{Definition der Operatoren K und Beweis einiger Eigenschaften} below, we will slightly generalize the proof of Theorem 1 		by Kostrykin and Makarov \cite{Kostrykin} to develop a proof of Theorem \ref{Diagonalisierung}.
	
	In order to derive Theorem \ref{Definition der Operatoren K und Beweis einiger Eigenschaften}, we first need to prove some auxiliary results.

	We denote by $ H^{2}(\mathbb{R}) $ the Hardy space on the line.	A Hankel operator on $ H^{2}(\mathbb{R}) $ with symbol $ \psi \in 											L^{\infty}(\mathbb{R}) $ acts as follows:
	\begin{equation*}
		S_{\psi} = \left. Q R M_{\psi} \right|_{H^{2}(\mathbb{R})}
	\end{equation*}
	where $ Q : L^{2}(\mathbb{R}) \rightarrow L^{2}(\mathbb{R}) $ is the orthogonal projection onto the closed subspace $ H^{2}(\mathbb{R}) $ of $ 					L^{2}(\mathbb{R}) $, $ R : L^{2}(\mathbb{R}) \rightarrow L^{2}(\mathbb{R}) $ denotes the reflection operator, $ (Rf)(x) = f(-x) $, and $ M_{\psi} : 		L^{2}(\mathbb{R}) \rightarrow L^{2}(\mathbb{R}) $ is the multiplication operator defined by $ M_{\psi}f = \psi \cdot f $.
	
	We use the following definition of the Fourier transform: if $ f \in L^{1}(\mathbb{R}) \cap L^{2}(\mathbb{R}) $ then
	\begin{equation*}
		(\mathrsfs{F} f)(x) = \frac{1}{\sqrt{2 \pi}} \int_{-\infty}^{\infty} f(t) \mathrm{e}^{-\mathrm{i} t x} \mathrm{d}t, \quad x \in \mathbb{R}.
	\end{equation*}
	
	Define
	\begin{equation*}
		\psi : \mathbb{R} \rightarrow \mathbb{R}, \quad \psi(t)=2 \cdot \mathds{1}_{\lbrack -1,1 \rbrack}(t),
	\end{equation*}
	and 
\begin{equation} \label{Definition der Funktion Phi}
	\phi : \mathbb{S}^{1} \rightarrow \mathbb{R}, \quad
	\phi(z) = 2 \cdot \mathds{1}_{\overline{\mathbb{H}_{r}} \cap \mathbb{S}^{1}}(z)
	= \begin{cases} \psi \left( \mathrm{i} \frac{1-z}{1+z} \right) & \text{if } z \neq -1 \\
		0 & \text{if } z = -1 \end{cases}.
\end{equation}

\begin{lemma} \label{Das erste Lemma dieses Abschnittes}
	Let $ \ell \in \mathbb{N}_{0} $. Then
	\begin{align*}
		\frac{1}{2 \pi} \int_{- \infty}^{\infty} t^{\ell} \psi(t) \mathrm{e}^{-\mathrm{i}tx} \mathrm{d}t 
		= \sqrt{\frac{2}{\pi}} \, \Big( \mathrsfs{F} \big\{ t^{\ell}~ \mathds{1}_{\lbrack -1,1 \rbrack}(t) \big\} \Big) (x)				
		= \frac{2}{\pi} ~ \mathrm{i}^{\ell} ~ \frac{\mathrm{d}^{\ell}}{\mathrm{d}x^{\ell}} \left\{	\frac{\mathrm{sin}(x)}{x} \right\}
	\end{align*} \label{Ableitung}
	holds true for all $ x \in \mathbb{R} $.
\end{lemma}

\begin{proof}
	This is straightforward.
\end{proof}

Let $ \ell \in \mathbb{N}_{0} $. Since 
\begin{equation} \label{Die explizite Funktion Psi als Symbol von S}
	\psi_{\ell}(t):= \left( \frac{1+\mathrm{i}t}{1-\mathrm{i}t} \right)^{\ell} \phi\left( \frac{1+\mathrm{i}t}{1-\mathrm{i}t} \right)
								 = \left( \frac{1+\mathrm{i}t}{1-\mathrm{i}t} \right)^{\ell} \psi(t)
\end{equation}
for all $ t \in \mathbb{R} $ we know by Power \cite{Power}, p.\,14, that the Hankel operator $ S_{\psi_{\ell}} $ on $ H^{2}(\mathbb{R}) $ is unitarily equivalent to the Hankel operator $ S_{z^{\ell +1} \phi} $ on $ H^{2} $.

We will compute the Fourier transform of $ t \mapsto \left( \frac{1+\mathrm{i}t}{1-\mathrm{i}t} \right)^{\ell} \psi(t) $ for all $ \ell \in \mathbb{N}_{0} $. Define $ \xi : \mathbb{R} \rightarrow \mathbb{R},~t \mapsto \frac{1}{1+t^{2}} $, and $ \tilde{\psi}_{\ell} : \mathbb{R} \rightarrow \mathbb{C},~t \mapsto (1+\mathrm{i}t)^{2 \ell} \cdot \psi(t) $. Since $ \frac{1+\mathrm{i}t}{1-\mathrm{i}t} = \frac{(1+\mathrm{i}t)^{2}}{1+t^{2}} $ we have to determine the Fourier transform of $ \xi^{\ell} \cdot \tilde{\psi}_{\ell} $.

We start by computing $ \mathrsfs{F} \tilde{\psi}_{\ell} $. Obviously, $ \tilde{\psi}_{\ell} \in L^{1}(\mathbb{R}) \cap L^{2}(\mathbb{R}) $ for all $ \ell \in \mathbb{N}_{0} $. Now the binomial theorem, the linearity of the Fourier transform, and Lemma \ref{Das erste Lemma dieses Abschnittes} imply the following representation for $ \mathrsfs{F} \tilde{\psi}_{\ell} $:

\begin{proposition} \label{FourierFaktor}
	Let $ \ell \in \mathbb{N}_{0} $. Then
	\begin{align*}
		\Big(\, \mathrsfs{F} \tilde{\psi}_{\ell} \Big)\, (w) = \sqrt{\frac{8}{\pi}} \sum\limits_{n=0}^{2 \ell} \binom{2 \ell}{n} (-1)^{n} 																																							\frac{\mathrm{d}^{n}}{\mathrm{d}w^{n}} ~ \frac{\mathrm{sin}(w)}{w}, \quad w \in \mathbb{R},
	\end{align*}
	holds true.
	The function $ z \mapsto \frac{\sin(z)}{z} $ is entire; if we restrict it to $ \mathbb{R} $ then the function $ z \mapsto \frac{\sin(z)}{z} $ and all 	of its derivatives of arbitrary order are bounded.
\end{proposition}

\begin{proof}
	It remains to show that the function $ z \mapsto \frac{\sin(z)}{z} = \sum_{n=0}^{\infty} (-1)^{n}~\frac{z^{2n}}{(2n+1)!} $ possesses the properties 		claimed above. It follows from the Cauchy-Hadamard theorem and Stirling's approximation that the function $ z \mapsto \frac{\sin(z)}{z} $ is entire.
	
	Let $ n \in \mathbb{N}_{0} $. Leibniz' rule implies that
	\begin{align*}
		\frac{\mathrm{d}^{n}}{\mathrm{d}x^{n}} ~ \frac{\sin(x)}{x}
		= \sum_{j=0}^{n} \binom{n}{j} ~ (-1)^{j} ~ j! ~ x^{-j-1} ~ \frac{\mathrm{d}^{n-j}}{\mathrm{d}x^{n-j}} \sin(x)
		\rightarrow 0
	\end{align*}
	as $ |x| \rightarrow \infty $ where $ x $ is real.
	Therefore, Proposition \ref{FourierFaktor} is proved.
\end{proof}

\begin{theorem} \label{Berechnung der Fourier-Transformierten von Xi hoch ell}
	For each $ \ell \in \mathbb{N} $ the Fourier transform of $ \xi^{\ell} $ is given by
	\begin{align*}
		\big(\, \mathrsfs{F} \{ \xi^{\ell} \} \big)\, (w)
		= \frac{1}{\sqrt{2 \pi}} \cdot \frac{\pi}{2^{\ell-1}} \cdot \mathrm{e}^{- \left| w \right|}~ p_{\ell}(|w|), \quad w \in \mathbb{R},~\text{where} \\
		p_{\ell}(w) = \sum\limits_{j=0}^{\ell-1} \frac{1}{2^{j}} \binom{\ell+j-1}{\ell-1} \frac{1}{(\ell-j-1)!} w^{\ell-j-1}, \quad w \in \mathbb{R}.
	\end{align*}
	\label{Koeffizienten}
\end{theorem}

\begin{remark}[to Theorem \ref{Berechnung der Fourier-Transformierten von Xi hoch ell}]
	There is a more general result in "Tabellen zur Fourier Transformation" \cite{Oberhettinger} by Oberhettinger; according to p.\,202 in 									\cite{Oberhettinger} one has
	\begin{align*}
		\int_{-\infty}^{\infty} \frac{\mathrm{e}^{\mathrm{i} xy}}{\{ a^{2} + (x \pm b)^{2} \}^{\nu}} \mathrm{d}x
		= 2 \mathrm{e}^{\mp \mathrm{i} b y} \frac{\sqrt{\pi}}{\Gamma(\nu)} \left( \frac{|y|}{2a} \right)^{\nu-1/2} K_{\nu-1/2}(a |y|), ~
		\operatorname{Re}(\nu) > 0,
	\end{align*}
	where $ \Gamma $ denotes the Gamma function and
	\begin{align*}
		&K_{\nu}(z) = \frac{\pi}{2} ~ \frac{I_{-\nu}(z) - I_{\nu}(z)}{\sin(\pi \nu)}, \quad
		I_{\nu}(z) = \mathrm{e}^{-\mathrm{i} \pi \nu / 2} J_{\nu}(z \mathrm{e}^{\mathrm{i} \pi / 2}), \\
		&J_{\nu}(z) = \sum\limits_{n=0}^{\infty} \frac{(-1)^{n} (z/2)^{\nu + 2n}}{n! ~ \Gamma(\nu+n+1)}.
	\end{align*}
	In case that $ a=1 $, $ b=0 $, and $ \nu = \ell \in \mathbb{N} $, we are in the situation of Theorem \ref{Berechnung der Fourier-Transformierten 				von Xi hoch ell}.
\end{remark}

We will need the following
\begin{lemma} \label{Faltungen}
	Let $ \ell \in \mathbb{N}_{0} $. Then
	\begin{align*}
		\int_{-\infty}^{\infty} \left| x \right|^{\ell} \mathrm{e}^{-\left| x \right|} \mathrm{e}^{-\left| y-x \right|} \mathrm{d}x 
		= \frac{\mathrm{e}^{-\left| y \right|}}{\ell+1} \bigg\{\, \left| y \right|^{\ell+1} + \sum\limits_{j=0}^{\ell-1}  																				\frac{(\ell+1)!}{(\ell-j)!} ~ \frac{ \left| y \right|^{\ell-j} }{2^{1+j}} + \frac{(\ell+1)!}{2^{\ell}} \bigg\}\,
	\end{align*}
	holds true for all $ y \in \mathbb{R} $.
\end{lemma}
\newpage
\begin{proof}
	In this proof we denote by $ \mathrm{H} $ the Heaviside function.
	
	Let $ \ell \in \mathbb{N}_{0} $, $ y \in \mathbb{R} \setminus \{0\} $. By elementary computations it follows that
	\begin{align*}
		I_{1} &:= \int_{0}^{\infty} x^{\ell}~\mathrm{e}^{-x}~\mathrm{e}^{-y+x}~\mathds{1}_{ \{ x<y \} }(x) \mathrm{d}x
						= \mathrm{e}^{-y}~\frac{y^{\ell+1}}{\ell+1} \cdot \mathrm{H}(y), \\
		I_{2} &:= \int_{0}^{\infty} x^{\ell} \mathrm{e}^{-x} \mathrm{e}^{y-x}~\mathds{1}_{ \{ x>y \} }(x) \mathrm{d}x \\
					&= \mathrm{e}^{-\left| y \right|} \bigg\{\, \frac{\ell !}{2^{\ell+1}} \cdot \mathrm{H}(-y) +\frac{1}{2^{\ell+1}}~ 																					\sum\limits_{j=0}^{\ell} (2y)^{\ell-j}~\frac{\ell !}{(\ell-j)!} \cdot \mathrm{H}(y) \bigg\}\,, \\
		I_{3} &:= \int_{-\infty}^{0} (-x)^{\ell} \mathrm{e}^{x} \mathrm{e}^{-y+x}~\mathds{1}_{ \{ x<y \} }(x) \mathrm{d}x \\
					&= \mathrm{e}^{-\left| y \right|} \bigg\{\, \frac{\ell !}{2^{\ell+1}} \cdot \mathrm{H}(y) + \frac{1}{2^{\ell+1}} ~ 																				\sum\limits_{j=0}^{\ell} (2y)^{\ell-j} ~ \frac{\ell !}{(\ell-j)!} \cdot \mathrm{H}(-y) \bigg\}\,, \\
		I_{4} &:= \int_{-\infty}^{0} (-x)^{\ell} \mathrm{e}^{x} \mathrm{e}^{y-x}~\mathds{1}_{ \{ x>y \} }(x) \mathrm{d}x
						= \mathrm{e}^{y}~\frac{(-y)^{\ell+1}}{\ell+1} \cdot \mathrm{H}(-y).
	\end{align*}	
	Altogether, we get that
	\begin{align*}
		&\int_{-\infty}^{\infty} \left| x \right|^{\ell} \mathrm{e}^{-\left| x \right|} \mathrm{e}^{-\left| y-x \right|} \mathrm{d}x
		= I_{1}+I_{2}+I_{3}+I_{4} \\
		&= \frac{\mathrm{e}^{-\left| y \right|}}{\ell+1} \bigg\{\, \left| y \right|^{\ell+1} + \sum\limits_{j=0}^{\ell-1} 2^{-1-j} \cdot 														\frac{(\ell+1)!}{(\ell-j)!} \cdot \left| y \right|^{\ell-j} + \frac{(\ell+1)!}{2^{\ell}} \bigg\}\,.
	\end{align*}
	Since $ \int\limits_{-\infty}^{\infty} \left| x \right|^{\ell} \mathrm{e}^{- \left| x \right|} \mathrm{e}^{-\left| 0-x \right|} 							\mathrm{d}x = 2 \int\limits_{0}^{\infty} x^{\ell} \mathrm{e}^{-2x} \mathrm{d}x = \frac{\ell !}{2^{\ell}} $ the proof is complete.
\end{proof}

We can now prove Theorem \ref{Koeffizienten}. 

\begin{proof}[Proof of Theorem \ref{Koeffizienten}]
	Obviously, $ \xi^{\ell} \in L^{1}(\mathbb{R}) \cap L^{2}(\mathbb{R}) $ for each $ \ell \in \mathbb{N} $. We prove Theorem \ref{Koeffizienten} by 				induction on $ \ell \in \mathbb{N} $. For $ \ell = 1 $ we use the inverse operator $ \mathrsfs{F}^{-1} $ to show that
	\begin{equation*}
		(\mathrsfs{F} \xi)(w) = \sqrt{\frac{\pi}{2}} \cdot \mathrm{e}^{- \left| w \right|}, \quad w \in \mathbb{R}.
	\end{equation*}
	Integration by parts easily leads to
	\begin{equation*}
		\bigg(\, \mathrsfs{F}^{-1} \bigg\{\, \sqrt{\frac{\pi}{2}} \mathrm{e}^{- \left| w \right| } \bigg\}\, \bigg)\,(t) = \int_{0}^{\infty} 														\mathrm{e}^{-w}~\mathrm{cos}(tw) \mathrm{d}w = \frac{1}{1+t^{2}}, \quad t \in \mathbb{R}.
	\end{equation*}
	Suppose that the claim holds up to $ \ell $. We have to show that it still holds for $ \ell+1 $. Taken together, the well-known behaviour of the 				Fourier transform of the convolution of two functions in $ L^{1}(\mathbb{R}) \cap L^{2}(\mathbb{R}) $, the induction hypothesis, Lemma 									\ref{Faltungen}, and a comparison of coefficients yield \newpage
	\begin{align*}
		& \big(\, \mathrsfs{F} \{ \xi^{\ell+1} \} \big)\, (w)
		= \frac{1}{\sqrt{2 \pi}} \Big(\, (\mathrsfs{F} \{ \xi^{\ell} \}) \ast (\mathrsfs{F} \xi) \Big)\, (w) \\
		&= \frac{1}{\sqrt{2 \pi}} ~ \frac{\pi}{2^{\ell}} \sum\limits_{j=0}^{\ell-1} \frac{1}{2^{j}} \binom{\ell+j-1}{\ell-1} \frac{1}{(\ell-j-1)!} \Big(\, ( 				\mathrm{e}^{- \left| x \right|} \left| x \right|^{\ell-j-1} ) \ast (\mathrm{e}^{-\left| x \right|} ) \Big)\, (w) \\
		&= \frac{1}{\sqrt{2 \pi}} \frac{\pi}{2^{\ell}} \mathrm{e}^{-\left| w \right|} \sum\limits_{j=0}^{\ell-1} \frac{1}{2^{j}} \binom{\ell+j-1}{\ell-1} 				\frac{1}{(\ell-j)!} \\
		& \quad \cdot \bigg\{\, \left| w \right|^{\ell-j} + \sum\limits_{m=0}^{\ell-j-2} 2^{-1-m} \frac{(\ell-j)!}{(\ell-1-j-m)!} \left| w 												\right|^{\ell-1-j-m} + \frac{(\ell-j)!}{2^{\ell-j-1}} \bigg\}\, \\
		&= \frac{\mathrm{e}^{-\left| w \right|}}{\sqrt{2 \pi}} ~ \frac{\pi}{2^{\ell}} ~ \bigg\{\, \frac{1}{2^{\ell-1}} \sum\limits_{j=0}^{\ell-1} 									\binom{\ell+j-1}{\ell-1} + \sum\limits_{\beta=1}^{\ell} \frac{1}{2^{\ell-\beta}} ~ \frac{1}{\beta !} \sum\limits_{\alpha=0}^{\ell-\beta} 								\binom{\ell+\alpha-1}{\ell-1} \left| w \right|^{\beta} \bigg\}\, \\
		&= \frac{1}{\sqrt{2 \pi}} ~ \frac{\pi}{2^{\ell}}~\mathrm{e}^{-\left| w \right|} \bigg\{\, \frac{1}{2^{\ell-1}} \binom{2\ell-1}{\ell} + 											\sum\limits_{\beta=1}^{\ell} \frac{1}{2^{\ell-\beta}} ~ \frac{1}{\beta !} \binom{2\ell-\beta}{\ell} \left| w \right|^{\beta} \bigg\}\, \\
		&= \frac{1}{\sqrt{2 \pi}} ~ \frac{\pi}{2^{\ell}} ~ \mathrm{e}^{-\left| w \right|} \sum\limits_{j=0}^{\ell} \frac{1}{2^{j}} \binom{\ell+j}{\ell} 						\frac{1}{(\ell-j)!} \left| w \right|^{\ell-j}, \quad w \in \mathbb{R};
	\end{align*}
	note that $ \frac{1}{2^{\ell}} \binom{2 \ell}{\ell} = \frac{1}{2^{\ell}} \left\{ \binom{2 \ell - 1}{\ell - 1} + \binom{2 \ell - 1}{\ell} \right\} = 		\frac{1}{2^{\ell - 1}} \binom{2 \ell - 1}{\ell} $. This finishes the proof of Theorem \ref{Koeffizienten}.
\end{proof}

Now Proposition \ref{FourierFaktor} and Theorem \ref{Koeffizienten} lead to
\begin{corollary} \label{Integralkerne}
	For the functions
	\begin{equation*}
		k^{(\ell)} : \mathbb{R} \rightarrow \mathbb{R}, \quad
		x \mapsto k^{(\ell)} (x)
		= \frac{1}{2 \pi} \int_{-\infty}^{\infty} \psi_{\ell}(t) \mathrm{e}^{-\mathrm{i}tx} \mathrm{d}t, \quad \ell \in \mathbb{N},
	\end{equation*}
	one has
	\begin{align*}
		k^{(\ell)}(x) &= \frac{1}{\sqrt{2 \pi}} ~ \Big(\, \mathrsfs{F} \big\{ \xi^{\ell} \cdot \tilde{\psi}_{\ell} \big\} \Big)\, (x) \\
		&= \frac{1}{\pi} ~ \frac{1}{2^{\ell-1}} \sum\limits_{n=0}^{2 \ell} (-1)^{n} ~\binom{2 \ell}{n} ~\frac{\mathrm{d}^{n}}{\mathrm{d}x^{n}} \bigg(\, 						\big(\, \mathrm{e}^{- \left| w \right|}~p_{\ell}(|w|) \big)\, \ast \Big(\, \frac{\mathrm{sin}(w)}{w} \Big)\, \bigg)\,(x).
	\end{align*}
\end{corollary}

We denote by $ \mathrsfs{D}(0,\infty) $ the set of all test functions on $ (0,\infty) $, i.\,e. the set of all infinitely differentiable functions on $ (0,\infty) $ with compact support.

\begin{theorem} \label{Definition der Operatoren K und Beweis einiger Eigenschaften}
	Let $ \ell \in \mathbb{N}_{0} $. Consider the integral operator $ K^{(\ell)} : L^{2}(0,\infty) \rightarrow L^{2}(0,\infty) $ with kernel function 
	$ k^{(\ell)} (x+y) $,
	\begin{equation*}
		\big( K^{(\ell)} f \big)\, (x) = \int_{0}^{\infty} k^{(\ell)}(x+y)~f(y) \mathrm{d}y, \quad
		f \in \mathrsfs{D}(0,\infty), \quad x > 0,
	\end{equation*}
	on the dense subspace $ \mathrsfs{D}(0,\infty) \subset L^{2}(0,\infty) $. Recall that $ k^{(0)}(x+y) = \frac{2}{\pi} ~ \frac{\sin(x+y)}{x+y} $. Then 		$ K^{(\ell)} $ is a bounded self-adjoint Hankel operator which is unitarily equivalent to the Hankel operator $ S_{\psi_{\ell}} $ with symbol 
	$ \psi_{\ell} $ on $ H^{2}(\mathbb{R}) $, and $ S_{\psi_{\ell}} $ is unitarily equivalent to the Hankel operator $ S_{z^{\ell +1} \phi} $ on 
	$ H^{2} $, for all $ \ell \in \mathbb{N}_{0} $. Here $ \phi $ denotes the function which is defined as in (\ref{Definition der 	Funktion Phi}).
\end{theorem}

\begin{proof}
	Let $ \ell \in \mathbb{N} $. According to Theorem \ref{Gesammelte Eigenschaften der Funktionen k} below, the function $ k^{(\ell)} $ is continuous on 
	$ (0,\infty) $, and one has $ k^{(\ell)}(x) = \mathcal{O}(x^{-1}) $ both as $ x \rightarrow 0+ $ and as $ x \rightarrow \infty $. Obviously, the 				function $ \psi_{\ell} = \xi^{\ell} \cdot \tilde{\psi}_{\ell} $ defined as in (\ref{Die explizite Funktion Psi als Symbol von S}) is 
	in $ L^{1}(\mathbb{R}) \cap L^{2}(\mathbb{R}) \cap L^{\infty}(\mathbb{R}) $, and according to Corollary \ref{Integralkerne} one has
	\begin{equation*}
		\frac{1}{\sqrt{2 \pi}} \mathrsfs{F} \psi_{\ell}
		= \frac{1}{\sqrt{2 \pi}} \mathrsfs{F} \{ \xi^{\ell} \cdot \tilde{\psi}_{\ell} \}
		= k^{(\ell)}
	\end{equation*}
	on the whole of $ \mathbb{R} $. It follows that the integral operator $ K^{(\ell)} $ with kernel function $ k^{(\ell)}(x+y) $ is the uniquely defined 	bounded Hankel operator on $ L^{2}(0,\infty) $ which is defined by
	\begin{equation*}
		\big( K^{(\ell)} f \big)\, (x) = \int_{0}^{\infty} k^{(\ell)}(x+y)~f(y) \mathrm{d}y, \quad
		f \in \mathrsfs{D}(0,\infty), \quad x > 0,
	\end{equation*}
	on the dense subspace $ \mathrsfs{D}(0,\infty) \subset L^{2}(0,\infty) $.
	According to Corollary \ref{Integralkerne} the function $ k^{(\ell)} $ is real-valued so $ \overline{k^{(\ell)}(y+x)} = k^{(\ell)}(x+y) $ for all $ 		x,y > 0 $. Therefore, $ K^{(\ell)} $ is self-adjoint. Now $ K^{(\ell)} $ is unitarily equivalent to $ S_{\psi_{\ell}} $ with symbol $ \psi_{\ell} $ on 	$ H^{2}(\mathbb{R}) $, and $ S_{\psi_{\ell}} $ is unitarily equivalent to $ S_{z^{\ell +1} \phi} $ on $ H^{2} $ by Power \cite{Power}, p.\,14. 
	
	In case that $ \ell = 0 $ the proof runs analogously.
\end{proof}

We denote by $ \ell_{+}^{2} $ the Hilbert space of all complex square-summable one-sided sequences $ x = (x_{0},x_{1},...) $.

\begin{lemma} \label{Fourierkoeffizienten}
	Let $ \phi $ be defined as in (\ref{Definition der Funktion Phi}) and
	\begin{equation*}
		c_{k} = \frac{1}{2 \pi} \int_{0}^{2 \pi} \mathrm{e}^{\mathrm{i}k \theta} \phi\left( \mathrm{e}^{\mathrm{i} \theta} \right) \mathrm{d}\theta
					= \frac{2}{\pi k} \mathrm{sin}(\pi k/2), \quad k \in \mathbb{N}.
	\end{equation*}	
	Then for each $ x \in \ell_{+}^{2} $
	\begin{equation*}
		\left( S_{z^{\ell +1} \phi} x \right)_{n} = \sum\limits_{k=0}^{\infty} c_{k+\ell+n+1} x_{k}, \quad
		\ell \in \mathbb{N}_{0},
	\end{equation*}
	holds true.
\end{lemma}

\begin{proof}
	This is straightforward.
\end{proof}

Let $ \mathcal{P}_{+} $ and $ \mathcal{P}_{-} $ be the orthogonal projections in $ \ell_{+}^{2} $ onto $ \mathcal{L}_{+} $ and $ \mathcal{L}_{-} $, respectively, where
\begin{align*}
	\mathcal{L}_{+} = \left\{ x \in \ell_{+}^{2} : x_{2k+1} = 0 \text{ for all } k \in \mathbb{N}_{0} \right\}, \\
	\mathcal{L}_{-} = \left\{ x \in \ell_{+}^{2} : x_{2k} = 0 \text{ for all } k \in \mathbb{N}_{0} \right\}.
\end{align*}

Let $ p \leq 1/2 $ be a real number. Consider the operators $ H_{p} : \ell_{+}^{2} \rightarrow \ell_{+}^{2} $ and $ \widetilde{H}_{p} : \ell_{+}^{2} \rightarrow \ell_{+}^{2} $ defined by
	\begin{equation*}
		\left( H_{p} x \right)_{n} = \sum\limits_{k=0}^{\infty} \frac{x_{k}}{1+k+n-p}
		~ \text{and} ~
		\left( \widetilde{H}_{p} x \right)_{n} = \sum\limits_{k=0}^{\infty} \frac{(-1)^{k+n}}{1+k+n-p} x_{k}, \quad x \in \ell_{+}^{2}, 
	\end{equation*}
respectively. Both the operators $ H_{p} $ and $ \widetilde{H}_{p} $ are bounded and self-adjoint, and one has $ V \widetilde{H}_{p} V = H_{p} $ where
\begin{equation} \label{Definition des unitaeren Operators V}
	V : \ell_{+}^{2} \rightarrow \ell_{+}^{2}, \quad \left( x_{0},x_{1},x_{2},x_{3},... \right) \mapsto \left( x_{0},-x_{1},x_{2},-x_{3},... \right),
\end{equation}
is unitary and $ V^{-1} = V $.

First, consider the case when $ \ell \in \mathbb{N}_{0} $ is even, i.\,e. $ \ell = 2m,~m \in \mathbb{N}_{0} $. With Lemma \ref{Fourierkoeffizienten}, it is easy to compute that
\begin{align*}
	\left( \mathcal{P}_{+} S_{z^{2m+1} \phi} \mathcal{P}_{+} x \right)_{2n} = \frac{1}{\pi} \sum\limits_{k=0}^{\infty} \frac{(-1)^{k+m+n}}{k+m+n+1/2} 																																								x_{2k}, \\
	\left( \mathcal{P}_{-} S_{z^{2m+1} \phi} \mathcal{P}_{-} x \right)_{2n+1} = -\frac{1}{\pi} \sum\limits_{k=0}^{\infty} \frac{(-1)^{k+m+n}}{k+m+n+3/2} 																																								x_{2k+1}, \\
	\mathcal{P}_{+} S_{z^{2m+1} \phi} \mathcal{P}_{-} = 0 = \mathcal{P}_{-} S_{z^{2m+1} \phi} \mathcal{P}_{+}.
\end{align*}
Define the operators
\begin{equation*}
	\mathcal{U}_{+} : \ell_{+}^{2} \rightarrow \mathcal{U}_{+}(\ell_{+}^{2}), \quad x = \left( x_{0},x_{1},x_{2},... \right) \mapsto \left( 																																																		x_{0},0,x_{1},0,x_{2},0,...	\right),
\end{equation*}

and
\begin{equation*}
	\mathcal{U}_{-} : \ell_{+}^{2} \rightarrow \mathcal{U}_{-}(\ell_{+}^{2}), \quad x = \left( x_{0},x_{1},x_{2},... \right) \mapsto 
		\left( 0,x_{0},0,x_{1},0,x_{2},0,... \right).
\end{equation*}
Both $ \mathcal{U}_{+} $ and $ \mathcal{U}_{-} $ are unitary, and it is easy to see that
\begin{align*}
	\mathcal{P}_{+} S_{z^{2m+1} \phi} \mathcal{P}_{+} \text{ leaves invariant } \mathcal{U}_{+}(\ell_{+}^{2}), \\
	\mathcal{P}_{-} S_{z^{2m+1} \phi} \mathcal{P}_{-} \text{ leaves invariant } \mathcal{U}_{-}(\ell_{+}^{2}).
\end{align*}
One has
\begin{align*}
	V \circ \mathcal{U}_{\pm}^{-1} \circ \left( \mathcal{P}_{\pm} S_{z^{2m+1} \phi} \mathcal{P}_{\pm} \right) \circ \mathcal{U}_{\pm} \circ V
	= \pm \frac{(-1)^{m}}{\pi} H_{\pm \frac{1}{2}-m}
\end{align*}
where $ V $ is defined as in (\ref{Definition des unitaeren Operators V}). Since the operator $ \mathcal{U} : \ell_{+}^{2} \rightarrow \mathcal{U}_{+}(\ell_{+}^{2}) \oplus \mathcal{U}_{-}(\ell_{+}^{2}) $ defined by $ \mathcal{U}x = \mathcal{P}_{+}x \oplus \mathcal{P}_{-}x $ is unitary we have shown:

\begin{lemma} \label{Hauptsatz Beweis Lemma I}
	For all $ m \in \mathbb{N}_{0} $ the operator $ S_{z^{2m+1} \phi} : \ell_{+}^{2} \rightarrow \ell_{+}^{2} $ is unitarily equivalent to the 							operator $ \left( \frac{(-1)^{m}}{\pi} H_{\frac{1}{2} - m} \right) \oplus \left( \frac{(-1)^{m+1}}{\pi} H_{-\frac{1}{2} - m} \right) : \ell_{+}^{2} 		\oplus \ell_{+}^{2} \rightarrow \ell_{+}^{2} \oplus \ell_{+}^{2} $.
\end{lemma}

Now consider the case when $ \ell \in \mathbb{N}_{0} $ is odd, i.\,e. $ \ell = 2m+1 $, $ m \in \mathbb{N}_{0} $. With Lemma \ref{Fourierkoeffizienten}, it is easy to check that
\begin{align*}
	\mathcal{P}_{+} S_{z^{2m+2} \phi} \mathcal{P}_{+} = 0 = \mathcal{P}_{-} S_{z^{2m+2} \phi} \mathcal{P}_{-}, \\
	\left( \mathcal{P}_{+} S_{z^{2m+2} \phi} \mathcal{P}_{-} x \right)_{2n} = -\frac{1}{\pi} \sum\limits_{k=0}^{\infty} \frac{(-1)^{k+m+n}}{k+m+n+3/2} 																																								x_{2k+1}, \\
	\left( \mathcal{P}_{-} S_{z^{2m+2} \phi} \mathcal{P}_{+} x \right)_{2n+1} = -\frac{1}{\pi} \sum\limits_{k=0}^{\infty} \frac{(-1)^{k+m+n}}{k+m+n+3/2} 																																								x_{2k}.
\end{align*}
Let $ x \oplus y \in \ell_{+}^{2} \oplus \ell_{+}^{2} $. One has
\begin{align*}
	&\begin{bmatrix}
		V & 0 \\
		0 & V
	\end{bmatrix}
	\begin{bmatrix}
		\mathcal{U}_{+}^{-1} & 0 \\
		0 & \mathcal{U}_{-}^{-1}
	\end{bmatrix}
	\begin{bmatrix}
		0 & \mathcal{P}_{+} S_{z^{2m+2} \phi} \mathcal{P}_{-} \\
		\mathcal{P}_{-} S_{z^{2m+2} \phi} \mathcal{P}_{+} & 0
	\end{bmatrix}
	\begin{bmatrix}
		\mathcal{U}_{+} & 0 \\
		0 & \mathcal{U}_{-}
	\end{bmatrix}
	\begin{bmatrix}
		V & 0 \\
		0 & V
	\end{bmatrix}
	\begin{bmatrix}
		x \\ y
	\end{bmatrix} \\
	&=
	\begin{bmatrix}
		0 & \frac{(-1)^{m+1}}{\pi} H_{-\frac{1}{2}-m} \\
		\frac{(-1)^{m+1}}{\pi} H_{-\frac{1}{2}-m} & 0
	\end{bmatrix}
	\begin{bmatrix}
		x \\ y
	\end{bmatrix}.
\end{align*}

We denote by $ I $ the identity operator $ \ell_{+}^{2} \rightarrow \ell_{+}^{2} $.
Obviously, the operator
\begin{equation*}
	\frac{1}{\sqrt{2}}
	\begin{bmatrix}
		I & -I \\
		I & I
	\end{bmatrix}
	: \ell_{+}^{2} \oplus \ell_{+}^{2} \rightarrow \ell_{+}^{2} \oplus \ell_{+}^{2}
\end{equation*}
is unitary and its inverse is given by
\begin{equation*}
	\frac{1}{\sqrt{2}}
	\begin{bmatrix}
		I & I \\
		-I & I
	\end{bmatrix}
	: \ell_{+}^{2} \oplus \ell_{+}^{2} \rightarrow \ell_{+}^{2} \oplus \ell_{+}^{2}.
\end{equation*}
Since
\begin{align*}
	&\frac{1}{\sqrt{2}}
	\begin{bmatrix}
		I & I \\
		-I & I
	\end{bmatrix}
	\begin{bmatrix}
		0 & \frac{(-1)^{m+1}}{\pi} H_{-\frac{1}{2}-m} \\
		\frac{(-1)^{m+1}}{\pi} H_{-\frac{1}{2}-m} & 0
	\end{bmatrix}
	\frac{1}{\sqrt{2}}
	\begin{bmatrix}
		I & -I \\
		I & I
	\end{bmatrix} \\
	&=
	\begin{bmatrix}
		\frac{(-1)^{m+1}}{\pi} H_{-\frac{1}{2}-m} & 0 \\
		0 & \frac{(-1)^{m}}{\pi} H_{-\frac{1}{2}-m}
	\end{bmatrix}
\end{align*}
we have shown:
\begin{lemma} \label{Hauptsatz Beweis Lemma II}
	For all $ m \in \mathbb{N}_{0} $ the operator $ S_{z^{2m+2} \phi} : \ell_{+}^{2} \rightarrow \ell_{+}^{2} $ is unitarily equivalent to the operator
	$ \left( \frac{(-1)^{m+1}}{\pi} H_{-\frac{1}{2}-m} \right) \oplus \left( \frac{(-1)^{m}}{\pi} H_{-\frac{1}{2}-m} \right) : \ell_{+}^{2} \oplus 					\ell_{+}^{2} \rightarrow \ell_{+}^{2} \oplus \ell_{+}^{2} $.	
\end{lemma}

Now Theorem \ref{Definition der Operatoren K und Beweis einiger Eigenschaften}, Lemma \ref{Hauptsatz Beweis Lemma I}, and Lemma \ref{Hauptsatz Beweis Lemma II} allow us to complete the proof of Theorem \ref{Diagonalisierung}. 

\begin{proof}[Proof of Theorem \ref{Diagonalisierung}]
	Let $ h : (0,\infty) \rightarrow (0,\infty) $ be defined as in Theorem \ref{Diagonalisierung}.
	Due to Rosenblum \cite{Rosenblum}, Theorem 4, the operator $ H_{p} $ is unitarily equivalent to the multiplication operator $ M_{h} $ on 
	$ L^{2} \Big(\, (0,\infty); \rho_{p}(\lambda) \mathrm{d}\lambda \Big)\, $ for all $ p \leq 1/2 $, in particular if $ p \in \big\{ \frac{1}{2}-m : m 		\in \mathbb{N}_{0} \big\}\, $. Therefore, the operators $ \frac{1}{\pi} H_{p} $ and $ -\frac{1}{\pi} H_{p} $ are unitarily equivalent to the 						multiplication operators $ M_{\tilde{h}} $ and $ M_{-\tilde{h}} $ on $ L^{2} \Big(\, (0,\infty); \rho_{p}(\lambda) \mathrm{d}\lambda \Big)\, $, 				respectively, where $ \tilde{h} := h / \pi $ and $ p \in \big\{ \frac{1}{2}-m : m \in \mathbb{N}_{0} \big\}\, $. From this, Theorem \ref{Definition 		der Operatoren K und Beweis einiger Eigenschaften}, and from the Lemmas \ref{Hauptsatz Beweis Lemma I} and \ref{Hauptsatz Beweis Lemma II}, all the 		assertions in Theorem \ref{Diagonalisierung} follow.
\end{proof}

\section{Some properties of the functions $ k^{(\ell)} $ for $ \ell \in \mathbb{N} $}

	If $ \ell = 0 $ then the function $ k^{(0)} : \mathbb{R} \rightarrow \mathbb{R} $ is given by $ k^{(0)}(x) = \frac{2}{\pi}~\frac{\sin(x)}{x} $; this 		function is fully understood. In this section, we will prove some properties of $ k^{(\ell)} $ for $ \ell \in \mathbb{N} $. First we need to prove 			some lemmas.

\begin{lemma}
	Let $ \ell \in \mathbb{N} $. Then one has
	\begin{enumerate}
		\item $ \sum\limits_{j=0}^{\ell-1} \frac{\cos \big(\, (\ell-j) \pi / 4 \big)\, }{2^{(\ell+j-2) / 2}} ~ \binom{\ell+j-1}{\ell-1} = 1 $.
		\item $ \sum\limits_{n=0}^{\ell} (-1)^{n} ~ \binom{2 \ell}{2n}
					= \cos(\ell \pi / 2) \cdot 2^{\ell} $.
		\item $ \sum\limits_{n=0}^{\ell-1} (-1)^{n+1} ~ \binom{2 \ell}{2n+1}
					= \sin(-\ell \pi / 2) \cdot 2^{\ell} $.
	\end{enumerate} \label{NuetzlicheSummen}
\end{lemma}
\newpage
\begin{proof}
	\begin{enumerate}
		\item Let $ \ell \in \mathbb{N} $. Define $ a_{\ell} := \sum\limits_{j=0}^{\ell-1} \frac{\mathrm{e}^{\mathrm{i} \cdot (\ell-j) \pi / 													4}}{2^{(\ell+j) / 2}} ~ \binom{\ell+j-1}{\ell-1} $. It is easy to compute
					\begin{align*}
						a_{\ell+1}
						&= \sum_{j=0}^{\ell} \frac{\mathrm{e}^{\mathrm{i} \cdot (\ell+1-j) \pi / 4}}{2^{(\ell+1+j) / 2}} ~ \binom{\ell+j-1}{\ell-1}
							+ \sum_{j=1}^{\ell} \frac{\mathrm{e}^{\mathrm{i} \cdot (\ell+1-j) \pi / 4}}{2^{(\ell+1+j) / 2}} ~ \binom{\ell+j-1}{\ell} \\
						&= a_{\ell} \cdot \frac{1 + \mathrm{i}}{2} + \frac{1 + \mathrm{i}}{2^{\ell+1}} ~ \binom{2 \ell - 1}{\ell-1}
							+ a_{\ell+1} \cdot \frac{1 - \mathrm{i}}{2} - \frac{1}{2^{\ell+1}} ~ \binom{2\ell}{\ell}
					\end{align*}
					which is equivalent to
					\begin{align*}
						a_{\ell+1}
						= a_{\ell} + \frac{\mathrm{i}}{2^{\ell+1}} ~ \binom{2\ell}{\ell}.
					\end{align*}
					From this it follows that
					\begin{align*}
						\frac{1}{2} = \operatorname{Re} \left( a_{1} \right) = ... = \operatorname{Re} \left( a_{\ell} \right)
						= \sum\limits_{j=0}^{\ell-1} \frac{\cos \big(\, (\ell-j) \pi / 4 \big)\, }{2^{(\ell+j) / 2}} ~ \binom{\ell+j-1}{\ell-1}.
					\end{align*}
		\item If $ \ell \in \mathbb{N} $ then
					\begin{equation*}
						\sum\limits_{n=0}^{\ell} (-1)^{n} ~ \binom{2 \ell}{2n}
						= \operatorname{Re} \left\{ (1 + \mathrm{i})^{2 \ell} \right\}
						= \cos(\ell \pi / 2) \cdot 2^{\ell}.
					\end{equation*} 
		\item If $ \ell \in \mathbb{N} $ then
					\begin{equation*}
						\sum\limits_{n=0}^{\ell-1} (-1)^{n+1} ~ \binom{2 \ell}{2n+1}
						= - \operatorname{Im} \left\{ (1 + \mathrm{i})^{2 \ell} \right\}
						= \sin(-\ell \pi / 2) \cdot 2^{\ell}.
					\end{equation*} 
	\end{enumerate}
\end{proof}

\begin{lemma}
	Let $ m \in \mathbb{N}_{0} $. Then one has
	\begin{enumerate}
		\item $ \int\limits_{-\infty}^{\infty} \mathrm{e}^{-|y|} ~ |y|^{m} ~ \sin(x-y) \mathrm{d}y
						= \frac{m!}{2^{(m-1) / 2}} ~ \cos \Big(\, (m+1) \pi / 4 \Big)\, \cdot \sin(x) $ \newline for all $ x \in \mathbb{R} $.
		\item $ \int\limits_{-\infty}^{\infty} \mathrm{e}^{-|y|} ~ |y|^{m} ~ \cos(x-y) \mathrm{d}y
						= \frac{m!}{2^{(m-1) / 2}} ~ \cos \Big(\, (m+1) \pi / 4 \Big)\, \cdot \cos(x) $ \newline for all $ x \in \mathbb{R} $.
	\end{enumerate} \label{FormelnZweierIntegrale}
\end{lemma}

\begin{proof}
	Let $ m \in \mathbb{N}_{0} $.
	\begin{enumerate}
		\item Let $ x \in \mathbb{R} $. It is straightforward to compute
					\begin{align*}
						\int_{-\infty}^{\infty} \mathrm{e}^{-|y|} ~ |y|^{m} ~ \sin(x-y) \mathrm{d}y
						&= 2 \sin(x) \int_{0}^{\infty} \mathrm{e}^{-y} y^{m} \cos(y) \mathrm{d}y \\
						&= \sin(x) ~ \left\{ \frac{m!}{2^{(m-1) / 2}} ~ \cos \Big(\, (m+1) \pi / 4 \Big)\, \right\}.																				
					\end{align*}
		\item Put $ \tilde{x} = x + \pi / 2 $, $ x \in \mathbb{R} $. Since $ \sin(\tilde{x}-y) = \cos(x-y) $ for every $ y \in \mathbb{R} $ the claim 						follows from the first part of the lemma.
	\end{enumerate}
\end{proof}

\begin{lemma}
	Let $ \ell \in \mathbb{N} $. Then
	\begin{align*}
		&\frac{1}{\pi} ~ \frac{1}{2^{\ell-1}}  \sum_{n=0}^{2 \ell} (-1)^{n} \binom{2\ell}{n} ~ \bigg(\, \big(\, \mathrm{e}^{- \left| w 														\right|}~p_{\ell}(|w|) \big)\, \ast \Big(\, \frac{\mathrm{d}^{n}}{\mathrm{d}w^{n}} ~ \mathrm{sin}(w) \Big)\, \bigg)\,(x) \\
		&= \frac{2}{\pi} \cdot \sin( x - \ell \pi / 2)
	\end{align*} \label{BestimmungAsymptotik}
	holds true for each $ x \in \mathbb{R} $.
\end{lemma}

\begin{proof}
	Let $ \ell \in \mathbb{N} $. It follows from Lemma \ref{FormelnZweierIntegrale} and Lemma \ref{NuetzlicheSummen} that
	\begin{align*}
		&\frac{1}{\pi} ~ \frac{1}{2^{\ell-1}} ~ \sum_{n=0}^{2 \ell} (-1)^{n} \binom{2\ell}{n} ~ \bigg(\, \big(\, \mathrm{e}^{- \left| w 														\right|}~p_{\ell}(|w|) \big)\, \ast \Big(\, \frac{\mathrm{d}^{n}}{\mathrm{d}w^{n}} ~ \mathrm{sin}(w) \Big)\, \bigg)\,(x) \\
		&= \frac{\sin(x)}{\pi \cdot 2^{\ell-1}} ~ \left[ \sum_{j=0}^{\ell-1} \frac{\cos \Big(\, (\ell-j) \pi / 4 \Big)\,}{2^{(\ell+j-2) / 2}} ~ 
			\binom{\ell + j - 1}{\ell - 1} \right]
			\cdot \left[ \sum_{\tilde{n}=0}^{\ell} (-1)^{\tilde{n}} \binom{2 \ell}{2 \tilde{n}} \right] \\
		& ~ + \frac{\cos(x)}{\pi \cdot 2^{\ell-1}} ~ \left[ \sum_{j=0}^{\ell-1} \frac{\cos \Big(\, (\ell-j) \pi / 4 \Big)\,}{2^{(\ell+j-2) / 2}} ~ 
			\binom{\ell + j - 1}{\ell - 1} \right]	\cdot \left[ \sum_{\tilde{n}=0}^{\ell-1} (-1)^{\tilde{n}+1} \binom{2 \ell}{2 \tilde{n} + 1} \right] \\
		&= \frac{2}{\pi} \cdot \sin(x - \ell \pi / 2), \quad x \in \mathbb{R}.
	\end{align*}
\end{proof}

We denote by $ \mathrsfs{C}^{\infty}(\mathbb{R}) $ the set of all infinitely differentiable functions on $ \mathbb{R} $.

\begin{theorem} \label{Gesammelte Eigenschaften der Funktionen k}
	Let $ \ell \in \mathbb{N} $. Then one has:
	\begin{enumerate}
		\item $ k^{(\ell)} \in \mathrsfs{C}^{\infty}(\mathbb{R}) \cap L^{2}(\mathbb{R}) \cap L^{\infty}(\mathbb{R}) $. All derivatives of arbitrary order of 					 $ k^{(\ell)} $ are in $ \mathrsfs{C}^{\infty}(\mathbb{R}) \cap L^{\infty}(\mathbb{R}) $. In particular, the limit $ \lim\limits_{x \rightarrow 						0+} k^{(\ell)}(x) $ exists in $ \mathbb{R} $.
		\item $ k^{(\ell)}(x) = \mathcal{O}(x^{-1}) $ as $ |x| \rightarrow \infty $, i.\,e. $ \limsup\limits_{|x| \rightarrow \infty} |x \cdot 												k^{(\ell)}(x)| \leq C < \infty $ \linebreak for some constant $ C > 0 $.
		\item $ \lim\limits_{|x| \rightarrow \infty} \left| x \cdot k^{(\ell)}(x) - \frac{2}{\pi} \cdot \sin(x - \ell \pi / 2) \right| = 0 $.
	\end{enumerate} \label{EigenschaftenKerne}
\end{theorem}

\begin{proof}
	Let $ \ell \in \mathbb{N} $.
	\begin{enumerate}
		\item We just show the boundedness. All the other assertions are obvious. Let $ m \in \mathbb{N}_{0} $. Define the following positive constants:
					\begin{align*}
						&C^{\prime \prime} := \max \left\{ \sup_{w \in \mathbb{R}} \left| \frac{\mathrm{d}^{j}}{\mathrm{d}w^{j}} ~ \frac{\sin(w)}{w} \right| ~ : j 																					\in \{ 0,1,...,2 \ell + m\} \right\}, \\
						&C^{\prime} := \max \left\{ \frac{1}{\pi} ~ \frac{\binom{\ell+j-1}{\ell-1}}{2^{j+\ell-1}} ~  \frac{\binom{2 \ell}{n}}{(\ell-j-1)!}  																			: j \in \{ 0,..., \ell-1 \}, n \in \{ 0,...,2 \ell \} \right\}, \\
						& C := C^{\prime} \cdot C^{\prime \prime}.
					\end{align*}
					\newpage One has
					\begin{align*}
						\left| \frac{\mathrm{d}^{m}}{\mathrm{d}x^{m}} k^{(\ell)}(x) \right|
						&\leq C \cdot \sum_{j=0}^{\ell-1} \sum_{n=0}^{2 \ell} ~ \int_{-\infty}^{\infty} \mathrm{e}^{- \left| y \right|}~\left| y 																	\right|^{\ell-j-1} \mathrm{d}y \\
						&= 2 (2 \ell + 1) C \cdot \sum_{j=0}^{\ell-1} (\ell - j - 1)!
					\end{align*}
					for all $ x \in \mathbb{R} $. Since this constant is finite and does not depend on $ x $ we have shown that each derivative of arbitrary 								order of $ k^{(\ell)} $ is in $ L^{\infty}(\mathbb{R}) $.
		\item This is immediate by the third part of the theorem.
		\item Let $ n \in \mathbb{N}_{0} $. The functions
					\begin{equation*}
						F_{2n} : z \mapsto z \cdot \frac{\mathrm{d}^{2n}}{\mathrm{d}z^{2n}} ~ \frac{\sin(z)}{z} - (-1)^{n} \sin(z)
					\end{equation*}
					and
					\begin{equation*}
						G_{2n+1} : z \mapsto z \cdot \frac{\mathrm{d}^{2n+1}}{\mathrm{d}z^{2n+1}} ~ \frac{\sin(z)}{z} - (-1)^{n} \cos(z)
					\end{equation*}
					are entire. Furthermore, Leibniz' rule implies that
		 			\begin{align*}
		 				x \cdot \frac{\mathrm{d}^{2n}}{\mathrm{d}x^{2n}} ~ \frac{\sin(x)}{x} - (-1)^{n} \sin(x)
		 				= \sum_{m=1}^{2n} \binom{2n}{m} (-1)^{m} ~ m! ~ x^{-m} ~ \frac{\mathrm{d}^{2n-m}}{\mathrm{d}x^{2n-m}} \sin(x)
		 			\end{align*}
		 			and
		 			\begin{align*}
		 				&x \cdot \frac{\mathrm{d}^{2n+1}}{\mathrm{d}x^{2n+1}} ~ \frac{\sin(x)}{x} - (-1)^{n} \cos(x) \\
		 				&= \sum_{m=1}^{2n+1} \binom{2n+1}{m} (-1)^{m} ~ m! ~ x^{-m} ~ \frac{\mathrm{d}^{2n+1-m}}{\mathrm{d}x^{2n+1-m}} \sin(x)
		 			\end{align*}
		 			tend to zero as $ |x| \rightarrow \infty $ where $ x $ is real. Therefore, $ F_{2n},G_{2n+1} \in \mathrsfs{C}^{\infty}(\mathbb{R}) \cap 									L^{\infty}(\mathbb{R}) $. It is easy to show that the dominated convergence theorem implies
		 			\begin{align*}
		 				\left\{ \begin{array}{l l}
							&\lim\limits_{|x| \rightarrow \infty} \int\limits_{-\infty}^{\infty} \mathrm{e}^{-|y|} p_{\ell}(|y|)  F_{2n}(x-y) \mathrm{d}y = 0, \\
							&\lim\limits_{|x| \rightarrow \infty} \int\limits_{-\infty}^{\infty} \mathrm{e}^{-|y|} p_{\ell}(|y|)  G_{2n+1}(x-y) \mathrm{d}y = 0,
						\end{array} \right. \tag{$ \ast $}
		 			\end{align*}
		 			for all $ \ell \in \mathbb{N} $.
		 			
		 			We already know that the function $ x \mapsto \frac{\mathrm{d}^{n}}{\mathrm{d}x^{n}} ~ \frac{\sin(x)}{x} $ is in $ 																			\mathrsfs{C}^{\infty}(\mathbb{R}) \cap L^{\infty}(\mathbb{R}) $ and that $  \frac{\mathrm{d}^{n}}{\mathrm{d}x^{n}} ~ \frac{\sin(x)}{x} $ 								$	\rightarrow 0 $ for each sequence of real numbers such that $ |x| \rightarrow \infty $. It is easy to show that the dominated convergence 						theorem implies
		 			\begin{equation*}
		 				\lim_{|x| \rightarrow \infty} \int_{-\infty}^{\infty} \mathrm{e}^{-|y|} p_{\ell}(|y|) \cdot y \cdot 																										\frac{\partial^{n}}{\partial x^{n}} ~ \frac{\sin(x-y)}{x-y} \mathrm{d}y
		 				= 0 \tag{$ \ast \ast $}
		 			\end{equation*}
		 			for all $ \ell \in \mathbb{N} $. Now Lemma \ref{BestimmungAsymptotik}, $ (\ast) $, and $ (\ast \ast) $ yield \newpage			
		 			\begin{align*}
		 				& \left| x \cdot k^{(\ell)}(x) - \frac{2}{\pi} \cdot \sin(x - \ell \pi / 2) \right| \\
		 				&= \left| \frac{1}{\pi} ~ \frac{1}{2^{\ell-1}} \sum\limits_{\tilde{n}=0}^{\ell} \binom{2 \ell}{2 \tilde{n}} \int_{-\infty}^{\infty} 								\mathrm{e}^{-|y|} ~ p_{\ell}(|y|) ~ F_{2 \tilde{n}}(x-y) \mathrm{d}y \right. \\
		 				& \quad - \frac{1}{\pi} ~ \frac{1}{2^{\ell-1}} \sum\limits_{\tilde{n}=0}^{\ell-1} \binom{2 \ell}{2 \tilde{n} + 1} 																					\int_{-\infty}^{\infty} \mathrm{e}^{-|y|} ~ p_{\ell}(|y|) ~ G_{2 \tilde{n} + 1}(x-y) \mathrm{d}y \\
		 				& \quad + \left. \frac{1}{\pi} ~ \frac{1}{2^{\ell-1}} \sum_{n=0}^{2 \ell} (-1)^{n} \binom{2 \ell}{n} \int_{-\infty}^{\infty} 												\mathrm{e}^{-|y|} ~ p_{\ell}(|y|) \cdot y \cdot \frac{\partial^{n}}{\partial x^{n}} ~ \frac{\sin(x-y)}{x-y} \mathrm{d}y \right| \\
		 				&\leq \frac{1}{\pi} ~ \frac{1}{2^{\ell-1}} \sum\limits_{\tilde{n}=0}^{\ell} \binom{2 \ell}{2 \tilde{n}} \left| ~																					\int_{-\infty}^{\infty} \mathrm{e}^{-|y|} ~ p_{\ell}(|y|) ~ F_{2 \tilde{n}}(x-y) \mathrm{d}y \right| \\
		 				& \quad + \frac{1}{\pi} ~ \frac{1}{2^{\ell-1}} \sum\limits_{\tilde{n}=0}^{\ell-1} \binom{2 \ell}{2 \tilde{n} + 1} \left| ~																		\int_{-\infty}^{\infty} \mathrm{e}^{-|y|} ~ p_{\ell}(|y|) ~ G_{2 \tilde{n} + 1}(x-y) \mathrm{d}y \right| \\
		 				& \quad + \frac{1}{\pi} ~ \frac{1}{2^{\ell-1}} \sum_{n=0}^{2 \ell} \binom{2 \ell}{n} \left| ~ \int_{-\infty}^{\infty} 															\mathrm{e}^{-|y|} ~ p_{\ell}(|y|) \cdot y \cdot \frac{\partial^{n}}{\partial x^{n}} ~ \frac{\sin(x-y)}{x-y} \mathrm{d}y \right| \\
		 				& \rightarrow 0			
		 			\end{align*}
		 			as $ |x| \rightarrow \infty $, for all $ \ell \in \mathbb{N} $.
	\end{enumerate}
\end{proof}

\begin{corollary} \label{Ergaenzung der gesammelten Eigenschaften der Funktionen k}
	Let $ \ell \in \mathbb{N} $. Then one has:
	\begin{enumerate}
		\item $ k^{(\ell)} \in L^{p}(\mathbb{R}) $ for all $ p \in (1,\infty] $.
		\item $ k^{(\ell)} \notin L^{1}(\mathbb{R}) $.
	\end{enumerate}
\end{corollary}

\begin{proof}
	Let $ \ell \in \mathbb{N} $.
	\begin{enumerate}
		\item We already know that $ k^{(\ell)} \in L^{\infty}(\mathbb{R}) $. If $ p \in (1,\infty) $ and $ N \in \mathbb{N} $, it is straightforward to 							compute
					\begin{align*}
						\int_{-\infty}^{\infty} \left| k^{(\ell)}(x) \right|^{p} \mathrm{d}x
						\leq C_{1} + \frac{2 C_{2}}{p-1} ~ N^{-p+1}
						< \infty
					\end{align*}
					where $ C_{1} := 2N \cdot \left\| k^{(\ell)} \right\|_{L^{\infty}(\mathbb{R})}^{p} < \infty $ and $ C_{2} := \sup\limits_{x \in \mathbb{R}} 						\left| x \cdot k^{(\ell)}(x) \right|^{p} < \infty $ by Theorem \ref{Gesammelte Eigenschaften der Funktionen k}.
		\item Denote by $ B_{\rho}(y) \subset \mathbb{R} $ the open ball of radius $ \rho > 0 $ around a real number $ y $. 
		
					Put $ \varepsilon := \frac{\pi}{4} > 0 $. Now it follows from the third part of Theorem \ref{EigenschaftenKerne} that
					\begin{equation*}
						\lim\limits_{n \rightarrow \infty} \frac{k^{(\ell)}(x_{n})}{\frac{2}{\pi} ~ \frac{\sin(x_{n}-\ell \pi / 2)}{x_{n}}} = 1 \quad
						\text{for each sequence } (x_{n})_{n \in \mathbb{N}} 
					\end{equation*}
					in
					$ \Omega := \mathbb{R} \setminus \bigcup\limits_{m \in \mathbb{Z}}  B_{\varepsilon} \big(\, (m + \ell / 2) \pi \big)\, $ such that
					$ \lim\limits_{n \rightarrow \infty} |x_{n}| = \infty $. Since the function $ x \mapsto \frac{k^{(\ell)}(x)}{\frac{2}{\pi} ~ \frac{\sin(x-\ell 					 \pi / 2)}{x}} $ is continuous on $ \Omega $ there exists some number $ N \in \mathbb{N} $ such that $ N \geq \ell \pi / 2 $ and
					\begin{equation*}
						\frac{k^{(\ell)}(x)}{\frac{2}{\pi} ~ \frac{\sin(x-\ell \pi / 2)}{x}} \geq 1 - \varepsilon > 0
					\end{equation*}
					for all $ x \in [N,\infty) \cap \Omega $. If we put $ \widetilde{\Omega} = \mathbb{R} \setminus \bigcup\limits_{m \in \mathbb{Z}} 											B_{\varepsilon}(m \pi) $, one has
					\begin{align*}
						\int_{-\infty}^{\infty} \left| k^{(\ell)}(x) \right| \mathrm{d}x
						&\geq (1-\varepsilon) ~ \frac{2}{\pi} ~ \int_{\Omega \cap (N,\infty)} \left| \frac{\sin(x-\ell \pi / 2)}{x} \right| \mathrm{d}x \\
						&= (1-\varepsilon) ~ \frac{2}{\pi} ~ \int_{\widetilde{\Omega} \cap (N+\ell \pi / 2,\infty)} \left| \frac{\sin(\tilde{x}-\ell 																																																	\pi)}{\tilde{x} - \ell \pi / 2} \right| \mathrm{d}\tilde{x} \\
						&\geq (1-\varepsilon) ~ \frac{2}{\pi} ~ \int_{\widetilde{\Omega} \cap (N+\ell \pi / 2,\infty)} \frac{|\sin(\tilde{x})|}{\tilde{x}} 																																															\mathrm{d}\tilde{x}
					\end{align*}
					where we substituted $ \tilde{x} := x + \ell \pi / 2 $. Since	$ N \geq \ell \pi / 2 $ and $ |\sin(x)| \geq \frac{1}{\sqrt{2}} $ for all $ x 						\in \widetilde{\Omega} $ this implies
					\begin{align*}
						\int_{-\infty}^{\infty} \left| k^{(\ell)}(x) \right| \mathrm{d}x
						&\geq (1-\varepsilon) ~ \frac{\sqrt{2}}{\pi} ~ \int_{\widetilde{\Omega} \cap (N \pi + \pi / 4,\infty)} \frac{1}{x} \mathrm{d}x \\
						&= (1-\varepsilon) ~ \frac{\sqrt{2}}{\pi} ~ \sum_{n=N}^{\infty} ~ \int_{n \pi + \pi / 4}^{n \pi + 3 \pi / 4} \frac{1}{x} \mathrm{d}x \\
						&\geq (1-\varepsilon) ~ \frac{\sqrt{2}}{2 \pi} ~ \sum_{n=N+1}^{\infty} \frac{1}{n} 
						= \infty.
					\end{align*}
	\end{enumerate}
\end{proof}

\section{Explicit computation of the functions $ k^{(\ell)} $ for $ \ell \in \mathbb{N} $}
  
  In this section, we will compute the functions $ k^{(\ell)} $ on $ (0,\infty) $. For this we need the \textit{exponential integral function} $ E_{1} $ 	 and the \textit{complementary exponential integral function} $ \mathrm{Ein} $. Recall that $ \mathrm{Ein} : \mathbb{C} \rightarrow \mathbb{C} $ is 			entire and can be represented as follows:
  \begin{align} \label{Darstellung der komplementaeren exponentiellen Integralfunktion}
  	\mathrm{Ein}(z)
  	= \int_{0}^{1} \frac{1 - \mathrm{e}^{-tz}}{t} \mathrm{d}t, \quad z \in \mathbb{C}.
  \end{align}
  It is well-known that the exponential integral function $ E_{1} $ is holomorphic on $ \mathbb{C}^{-} := \mathbb{C} \setminus (-\infty,0] $ and can be 	represented on $ \overline{\mathbb{H}_{r}} \setminus \{ 0 \} $ as follows:
 	\begin{equation} \label{Darstellung der exponentiellen Integralfunktion}
		E_{1}(z) =  \int_{1}^{\infty} \frac{\mathrm{e}^{-tz}}{t} \mathrm{d}t, \quad z \in \overline{\mathbb{H}_{r}} \setminus \{ 0 \}.
	\end{equation}
   
  Furthermore, if $ z $ is in $ \mathbb{C}^{-} $ then
  \begin{equation} \label{RelationEI}
		E_{1}(z) = \mathrm{Ein}(z) - \log(z) - \gamma
	\end{equation}
  where $ \log $ is the principal value of the complex logarithm and $ \gamma $ denotes Euler's constant. 
  
  First we will prove some lemmas. \newpage

	According to \cite{Erdelyi}, p.\,217, formula (14), we know the following fact:
	\begin{lemma} \label{ExplitziteBerechnung Zitat H.T.F.}
		Let $ a $ be in $ \mathbb{C} $, $ \operatorname{Re}(a) > 0 $, and $ n \in \mathbb{N}_{0} $. Then the following formula holds true for all $ x > 0 $:
		\begin{align*}
			\int_{0}^{\infty} y^{n} \mathrm{e}^{-ay} \frac{1}{x+y} \mathrm{d}y
			= (-1)^{n} x^{n} \mathrm{e}^{ax} E_{1}(ax) + \sum\limits_{r=1}^{n} (-1)^{n-r} (r-1)! ~ a^{-r} x^{n-r}.
		\end{align*}		
	\end{lemma}

  We will also need
   
  \begin{lemma} \label{ExpliziteBerechnungI}
	The following formulas hold true:
	\begin{enumerate}
		\item For all $ x > 0 $,
					\begin{equation*}
						\int_{0}^{\infty} \mathrm{e}^{-y} ~ \frac{\sin(x-y)}{x-y} \mathrm{d}y
						= \pi \mathrm{e}^{-x} + \frac{\mathrm{i}}{2} \mathrm{e}^{-x} \cdot \left\{ E_{1}(-x - \mathrm{i}x) - E_{1}(-x + \mathrm{i}x) \right\}.
					\end{equation*}
		\item Let $ m \in \mathbb{N} $. If $ r \in \{ 1,...,m \} $ then 
					\begin{equation*}
						\sum\limits_{j=r}^{m} (-1)^{j} \binom{m}{j} \frac{(j-1)!}{(j-r)!} = (-1)^{r} (r-1)!.
					\end{equation*}
		\item If $ m \in \mathbb{N}_{0} $ then
					\begin{align*}
						&\int_{0}^{\infty} \mathrm{e}^{-y} y^{m} ~ \frac{\sin(x-y)}{x-y} \mathrm{d}y \\
						&= \bigg(\, \pi \mathrm{e}^{-x} + \frac{\mathrm{i}}{2} \mathrm{e}^{-x} \cdot \left\{ E_{1}(-x - \mathrm{i}x) - E_{1}(-x + \mathrm{i}x) 											\right\} \bigg)\, \cdot x^{m} \\
						& \quad + \sum\limits_{r=1}^{m} \frac{(r-1)!}{2^{r/2}} x^{m-r} \sin \Big(\, \frac{r \pi}{4} - x \Big)\,, \quad x > 0.
					\end{align*}
	\end{enumerate}
\end{lemma}

\begin{proof}
	\begin{enumerate}
		\item From (\ref{Darstellung der komplementaeren exponentiellen Integralfunktion}) and (\ref{RelationEI}) it follows that
					\begin{align*}
						&\int_{-1}^{1} \frac{\mathrm{e}^{-y(x-\mathrm{i}x)} - \mathrm{e}^{-y(x+\mathrm{i}x)}}{y} \mathrm{d}y \\
						&= - \mathrm{Ein}(x-\mathrm{i}x) + \mathrm{Ein}(-x+\mathrm{i}x) + \mathrm{Ein}(x+\mathrm{i}x) - \mathrm{Ein}(-x-\mathrm{i}x) \\
						&= - E_{1}(x-\mathrm{i}x) + E_{1}(-x+\mathrm{i}x) + E_{1}(x+\mathrm{i}x) - E_{1}(-x-\mathrm{i}x)
							 + 2 \pi \mathrm{i}, \quad x > 0.
					\end{align*}
					If we substitute $ \hat{y} := -1 + y / x $ and then apply (\ref{Darstellung der exponentiellen Integralfunktion}) we get that
					\begin{align*}
						&\int_{0}^{\infty} \mathrm{e}^{-y} ~ \frac{\sin(x-y)}{x-y} \mathrm{d}y
						= \frac{\mathrm{e}^{-x}}{2 \mathrm{i}} \int_{-1}^{\infty} \left( - \mathrm{e}^{-\hat{y}(x+\mathrm{i}x)} + 																				 	 \mathrm{e}^{-\hat{y}(x-\mathrm{i}x)} \right) \frac{1}{\hat{y}} ~ \mathrm{d}\hat{y} \\
						& \qquad = \frac{\mathrm{e}^{-x}}{2 \mathrm{i}} \left\{ E_{1}(x-\mathrm{i}x) - E_{1}(x+\mathrm{i}x) + \int_{-1}^{1} 																				\frac{\mathrm{e}^{-\hat{y}(x-\mathrm{i}x)} - \mathrm{e}^{-\hat{y}(x+\mathrm{i}x)}}{\hat{y}} \mathrm{d}\hat{y} \right\} \\
						& \qquad = \pi \mathrm{e}^{-x} + \frac{\mathrm{i}}{2} \mathrm{e}^{-x} \cdot \left\{ E_{1}(-x-\mathrm{i}x) - E_{1}(-x+\mathrm{i}x) \right\},
							\quad x > 0.
					\end{align*}
		\item It is not hard to show
					\begin{equation*}
						\sum\limits_{j=r}^{m} (-1)^{j} \binom{m}{j} \binom{j-1}{r-1} = (-1)^{r} \quad \text{for all } r \in \{ 1,...,m \} 
					\end{equation*}
					by induction on $ m \in \mathbb{N} $. \newpage
%					Obviously, if $ m = 1 $ then $ r = 1 $ and the induction basis holds. Suppose that the claim holds up to $ m $. Clearly $ 										%					\sum\limits_{j=m+1}^{m+1} (-1)^{j} \binom{m+1}{j} \binom{j-1}{m} = (-1)^{m+1} $. If $ 1 \leq r \leq m $ then it follows from $ 								%					\binom{m+1}{j} = \binom{m}{j-1} + \binom{m}{j} $, the induction hypothesis, and
%					\begin{align*}
%						&\sum\limits_{j=r}^{m+1} (-1)^{j} \binom{m}{j-1} \binom{j-1}{r-1} \\
%						&= \frac{(-1)^{r}}{(m-r+1)!} ~ \frac{m!}{(r-1)!} \sum\limits_{k=0}^{m-r+1} (-1)^{k} \binom{m-r+1}{k} 
%						= 0
%					\end{align*}
%					that
%					\begin{align*}
%						& \sum\limits_{j=r}^{m+1} (-1)^{j} \binom{m+1}{j} \binom{j-1}{r-1} \\
%						&= \sum\limits_{j=r}^{m+1} (-1)^{j} \binom{m}{j-1} \binom{j-1}{r-1} + \sum\limits_{j=r}^{m} (-1)^{j} \binom{m}{j} \binom{j-1}{r-1} 
%						= (-1)^{r}.
%					\end{align*}
		\item If $ m = 0 $ then the assertion follows immediately from (1).
					
					Now let $ m $ be in $ \mathbb{N} $.
					Let $ a $ be in $ \mathbb{C} $ such that $ \operatorname{Re}(a) > 0 $. Then the substitution $ r := 1+k $, a comparison of coefficients, and 						the second part of this lemma imply
					\begin{align*}
						&x^{m} \sum\limits_{j=1}^{m} \binom{m}{j} (-1)^{j} \sum\limits_{k=0}^{j-1} \binom{j-1}{k} (-1)^{k} x^{-1-k} \int_{0}^{\infty} 												\mathrm{e}^{-ay} y^{k} \mathrm{d}y \\
						&= x^{m} \sum\limits_{j=1}^{m} \binom{m}{j} (-1)^{j} \sum\limits_{k=0}^{j-1} \binom{j-1}{k} (-1)^{k} \frac{k!}{a^{1+k}} x^{-1-k} \\
						&= x^{m} \sum\limits_{r=1}^{m} a^{-r} x^{-r} (-1)^{r+1} \sum\limits_{j=r}^{m} (-1)^{j} \binom{m}{j} \frac{(j-1)!}{(j-r)!} \\
						&= - \sum\limits_{r=1}^{m} a^{-r} (r-1)! ~ x^{m-r}, \quad x > 0. \tag{$ \ast $}
					\end{align*}
					From the binomial theorem, ($ \ast $), and the first part of this lemma it follows that
					\begin{align*}
						&\int_{0}^{\infty} \mathrm{e}^{-y} \{ (y-x) + x \} ^{m} ~ \frac{\sin(x-y)}{x-y} \mathrm{d}y \\
						&= x^{m} \int_{0}^{\infty} \mathrm{e}^{-y} ~ \frac{\sin (x-y)}{x-y} \mathrm{d}y \\
						& \quad	+ \frac{\mathrm{e}^{\mathrm{i} x}}{2 \mathrm{i}} x^{m} \sum\limits_{j=1}^{m} \binom{m}{j} (-1)^{j} \sum\limits_{k=0}^{j-1} 												\binom{j-1}{k} \frac{(-1)^{k}}{x^{k+1}} \int_{0}^{\infty} \mathrm{e}^{-(1 + \mathrm{i}) y} y^{k} \mathrm{d}y \\
						& \quad - \frac{\mathrm{e}^{- \mathrm{i}x}}{2 \mathrm{i}} x^{m} \sum\limits_{j=1}^{m} \binom{m}{j} (-1)^{j} \sum\limits_{k=0}^{j-1} 											\binom{j-1}{k} \frac{(-1)^{k}}{x^{k+1}} \int_{0}^{\infty} \mathrm{e}^{-(1 - \mathrm{i}) y} y^{k} \mathrm{d}y \\
						&= \bigg(\, \pi \mathrm{e}^{-x} + \frac{\mathrm{i}}{2} \mathrm{e}^{-x} \cdot \left\{ E_{1}(-x - \mathrm{i}x) - E_{1}(-x + \mathrm{i}x) 											\right\} \bigg)\, \cdot x^{m} \\
						& \quad - \frac{\mathrm{e}^{\mathrm{i} x}}{2 \mathrm{i}} \sum\limits_{r=1}^{m} (1 + \mathrm{i})^{-r} (r-1)! ~ x^{m-r}
							+ \frac{\mathrm{e}^{- \mathrm{i}x}}{2 \mathrm{i}} \sum\limits_{r=1}^{m} (1 - \mathrm{i})^{-r} (r-1)! ~ x^{m-r} \\
						&= \bigg(\, \pi \mathrm{e}^{-x} + \frac{\mathrm{i}}{2} \mathrm{e}^{-x} \cdot \left\{ E_{1}(-x - \mathrm{i}x) - E_{1}(-x + \mathrm{i}x) 											\right\} \bigg)\, \cdot x^{m} \\
						& \quad + \sum\limits_{r=1}^{m} \frac{(r-1)!}{2^{r/2}} x^{m-r} \sin \Big(\, \frac{r \pi}{4} - x \Big)\,, \quad x > 0,
					\end{align*}
					as was to be shown.	
	\end{enumerate}		
\end{proof}
\newpage
\begin{lemma} \label{ExpliziteBerechnungII}
	Let $ m,n \in \mathbb{N}_{0} $ such that $ m \geq n $. Then one has:
	\begin{enumerate}
		\item For all $ x > 0 $,
					\begin{align*}
						&P_{m,n}^{-}(x)
						:= \int_{0}^{\infty} \mathrm{e}^{-y} y^{m} \frac{\partial^{n}}{\partial x^{n}} ~ \frac{\sin(x-y)}{x-y} \mathrm{d}y \\
						&= \left( \pi \mathrm{e}^{-x} + \frac{\mathrm{i}}{2} \mathrm{e}^{-x} \left\{ E_{1}(-x-\mathrm{i}x) - E_{1}(-x+\mathrm{i}x) \right\} \right) 							\sum\limits_{j=0}^{n} \binom{n}{j} \frac{(-1)^{n-j} ~ m!}{(m-j)!} x^{m-j} \\
						& \quad + \sum\limits_{j=0}^{n} \binom{n}{j} (-1)^{n-j} \frac{m!}{(m-j)!} \sum\limits_{r=1}^{m-j}  
									\frac{(r-1)!}{2^{r/2}} \sin \Big(\, \frac{r \pi}{4} - x \Big)\, x^{m-j-r}.
					\end{align*}			
		\item For all $ x > 0 $,
					\begin{align*}
						&P_{m,n}^{+}(x)
						:= \int_{0}^{\infty} \mathrm{e}^{-y} y^{m} \frac{\partial^{n}}{\partial x^{n}} ~ \frac{\sin(x+y)}{x+y} \mathrm{d}y \\
						&= (-1)^{m} \frac{\mathrm{i}}{2} \mathrm{e}^{x} \left\{ E_{1}(x+\mathrm{i}x) - E_{1}(x-\mathrm{i}x) \right\} 
							\sum\limits_{j=0}^{n} \binom{n}{j} \frac{m!}{(m-j)!} x^{m-j} \\
						& \quad + \sum\limits_{j=0}^{n} \binom{n}{j} \frac{m!}{(m-j)!} \sum\limits_{r=1}^{m-j} (-1)^{m-r} \frac{(r-1)!}{2^{r/2}} \sin \Big(\, 										\frac{r \pi}{4} + x \Big)\, x^{m-j-r}.
					\end{align*}
	\end{enumerate}
\end{lemma}

\begin{proof}
	\begin{enumerate}
		\item It is straightforward to compute
					\begin{align*}
						&\int_{0}^{\infty} \mathrm{e}^{-y} y^{m} \frac{\partial^{n}}{\partial x^{n}} ~ \frac{\sin(x-y)}{x-y} \mathrm{d}y
						= \int_{0}^{\infty} \frac{\partial^{n}}{\partial y^{n}} \{ \mathrm{e}^{-y} y^{m} \} ~ \frac{\sin(x-y)}{x-y} \mathrm{d}y \\
						&= \sum\limits_{j=0}^{n} \binom{n}{j} (-1)^{n-j} \frac{m!}{(m-j)!} \int_{0}^{\infty} \mathrm{e}^{-y} y^{m-j} ~ \frac{\sin(x-y)}{x-y} 								\mathrm{d}y, \quad x > 0, 
					\end{align*}
					where we used integration by parts and Leibniz' rule.
					Now the claim follows from the third part of Lemma \ref{ExpliziteBerechnungI}.
		\item This is proved analogously to the first part of this lemma; just use Lemma \ref{ExplitziteBerechnung Zitat H.T.F.} instead of Lemma 										\ref{ExpliziteBerechnungI}.
	\end{enumerate}
\end{proof}

\begin{lemma} \label{ExpliziteBerechnungIII}
	Let $ m \in \mathbb{N} $ and $ n \in \mathbb{N}_{0} $. Then the following results hold true:
	\begin{enumerate}
		\item For all $ x > 0 $,
			\begin{equation*}
				\int\limits_{0}^{\infty} \mathrm{e}^{-y} \frac{\partial^{n}}{\partial y^{n}} ~ \frac{\sin(x-y)}{x-y} \mathrm{d}y
				= \sum\limits_{b=0}^{n-1} (-1)^{n+b} \frac{\mathrm{d}^{n-1-b}}{\mathrm{d}x^{n-1-b}} ~ \frac{\sin(x)}{x}
					+ \int\limits_{0}^{\infty} \mathrm{e}^{-y} ~ \frac{\sin(x-y)}{x-y} \mathrm{d}y.
			\end{equation*}
		\item For all $ x > 0 $,
			\begin{equation*}
				\int_{0}^{\infty} \mathrm{e}^{-y} ~ \frac{\partial^{n}}{\partial y^{n}} ~ \frac{\sin(x+y)}{x+y} \mathrm{d}y
					= - \sum\limits_{b=0}^{n-1} \frac{\mathrm{d}^{n-1-b}}{\mathrm{d}x^{n-1-b}} ~ \frac{\sin(x)}{x}
						+ \int_{0}^{\infty} \mathrm{e}^{-y} ~ \frac{\sin(x+y)}{x+y} \mathrm{d}y.
			\end{equation*} \newpage
		\item If $ m < n $ then
					\begin{align*}
						 &Q_{m,n}^{-}(x)
						:=	\int_{0}^{\infty} \mathrm{e}^{-y} y^{m} \frac{\partial^{n}}{\partial x^{n}} ~ \frac{\sin(x-y)}{x-y} \mathrm{d}y \\
						&= (-1)^{n-m-1} \sum\limits_{s=0}^{n-m-1} (-1)^{s} \frac{(n-s-1)!}{(n-m-s-1)!} ~ \frac{\mathrm{d}^{s}}{\mathrm{d}x^{s}} ~ \frac{\sin(x)}{x} 							\\ 
						& \quad + \left( \pi \mathrm{e}^{-x} + \frac{\mathrm{i}}{2} \mathrm{e}^{-x} \left\{ E_{1}(-x-\mathrm{i}x) - E_{1}(\mathrm{i}x-x) \right\} 										\right) \sum\limits_{j=0}^{m} \binom{n}{j} \frac{(-1)^{n-j} ~ m!}{(m-j)!} x^{m-j} \\
						& \quad + \sum\limits_{j=0}^{m} \binom{n}{j} (-1)^{n-j} \frac{m!}{(m-j)!} \sum\limits_{r=1}^{m-j}  
									\frac{(r-1)!}{2^{r/2}} \sin \Big(\, \frac{r \pi}{4} - x \Big)\, x^{m-j-r}, \quad x > 0.
					\end{align*} 
		\item If $ m < n $ then
					\begin{align*}
						&Q_{m,n}^{+}(x)
						:= \int_{0}^{\infty} \mathrm{e}^{-y} y^{m} \frac{\partial^{n}}{\partial x^{n}} ~ \frac{\sin(x+y)}{x+y} \mathrm{d}y \\
						&= (-1)^{m+1} \sum\limits_{s=0}^{n-m-1} \frac{(n-s-1)!}{(n-m-s-1)!} ~ \frac{\mathrm{d}^{s}}{\mathrm{d}x^{s}} ~ \frac{\sin(x)}{x} \\
						& \quad + (-1)^{m} \frac{\mathrm{i}}{2} \mathrm{e}^{x} \left\{ E_{1}(x+\mathrm{i}x) - E_{1}(x-\mathrm{i}x) \right\} \sum\limits_{j=0}^{m} 								\binom{n}{j} \frac{m!}{(m-j)!} x^{m-j} \\
						& \quad + \sum\limits_{j=0}^{m} \binom{n}{j} \frac{m!}{(m-j)!} \sum\limits_{r=1}^{m-j} (-1)^{m-r}  
								\frac{(r-1)!}{2^{r/2}} \sin \Big(\, \frac{r \pi}{4} + x \Big)\, x^{m-j-r}, \quad x > 0.
					\end{align*}
	\end{enumerate}
\end{lemma}

\begin{proof}
	\begin{enumerate}
		\item This is easily proved by induction on $ n \in \mathbb{N}_{0} $.
		\item This is easily proved by induction on $ n \in \mathbb{N}_{0} $, too.
		\item It follows analogously to the proof of the first part of Lemma \ref{ExpliziteBerechnungII} that
					\begin{align*}
						&\int_{0}^{\infty} \mathrm{e}^{-y} y^{m} \frac{\partial^{n}}{\partial x^{n}} ~ \frac{\sin(x-y)}{x-y} \mathrm{d}y \\
						&= \mathrm{e}^{-x} \left( \pi + \frac{\mathrm{i}}{2} \left\{ E_{1}(-x-\mathrm{i}x) - E_{1}(\mathrm{i}x -x ) \right\} \right) 															\sum\limits_{j=0}^{m} \binom{n}{j} \frac{(-1)^{n-j} m!}{(m-j)!} x^{m-j} \\
						& \quad + \sum\limits_{j=0}^{m} \binom{n}{j} (-1)^{n-j} \frac{m!}{(m-j)!} \sum\limits_{r=1}^{m-j}  
									\frac{(r-1)!}{2^{r/2}} \sin \Big(\, \frac{r \pi}{4} - x \Big)\, x^{m-j-r} \\
						& \quad + \text{boundary terms} \tag{$ \ast $} 
					\end{align*}
					for all $ x > 0 $. Since
					\begin{equation*}
						\int_{0}^{\infty} \mathrm{e}^{-y} \frac{\partial^{j}}{\partial y^{j}} ~ \frac{\sin(x-y)}{x-y} \mathrm{d}y
						= \int_{0}^{\infty} \mathrm{e}^{-y} \frac{\partial^{j-1}}{\partial y^{j-1}} ~ \frac{\sin(x-y)}{x-y} \mathrm{d}y
							+ \text{boundary term}
					\end{equation*}
					and
					\begin{equation*}
						\frac{\mathrm{d}^{m+j}}{\mathrm{d} y^{m+j}} \left\{ \mathrm{e}^{-y} y^{m} \right\}
						= m! \, (-1)^{j} \binom{m+j}{m} \mathrm{e}^{-y} + \mathrm{remainder} 
					\end{equation*}
					for all $ j \in \mathbb{N} $ it follows with (1) that the boundary terms in ($ \ast $) are given by
					\begin{align*}
						 &m! \sum\limits_{a=0}^{n-m-1} \binom{m+a-1}{m-1} (-1)^{a} \sum\limits_{b=0}^{n-m-a-1} (-1)^{b} 																														\frac{\mathrm{d}^{n-m-a-1-b}}{\mathrm{d}x^{n-m-a-1-b}} ~ \frac{\sin(x)}{x} \\
						 &= m! \sum\limits_{r=0}^{n-m-1} (-1)^{r} \frac{\mathrm{d}^{n-m-r-1}}{\mathrm{d}x^{n-m-r-1}} ~ \frac{\sin(x)}{x} \sum\limits_{a=0}^{r} 											\binom{m+a-1}{m-1} \\
						 &= (-1)^{n-m-1} \sum\limits_{s=0}^{n-m-1} (-1)^{s} \frac{(n-s-1)!}{(n-m-s-1)!} ~ \frac{\mathrm{d}^{s}}{\mathrm{d}x^{s}} ~ \frac{\sin(x)}{x},
						 \quad x > 0,
					\end{align*}
					where we substituted $ r := a+b $, compared coefficients, and put $ s := n-m-r-1 $. Now the claim follows.
		\item This is proved analogously to the third part of this lemma; just use (2) instead of (1).
	\end{enumerate}
\end{proof}

\begin{theorem} \label{Satz ueber die explizite Darstellung der Funktionen k}
	Let $ \ell \in \mathbb{N} $. Then one has
	\begin{align*}
		&k^{(\ell)}(x) \\
		&= \frac{1}{\pi} ~ \frac{1}{2^{2\ell - 2}} \left\{ \sum\limits_{m=0}^{\ell-1} \, \sum\limits_{n=0}^{m} (-1)^{n} \binom{2\ell}{n} \frac{2^{m}}{m!} 					\binom{2\ell - m - 2}{\ell - 1} \left[ P_{m,n}^{-}(x) + P_{m,n}^{+}(x) \right] \right. \\
		& \quad	\left. + \sum\limits_{m=0}^{\ell-1} \, \sum\limits_{n=m+1}^{2\ell} (-1)^{n} \binom{2\ell}{n} \frac{2^{m}}{m!} \binom{2\ell - m - 2}{\ell - 				1} \left[ Q_{m,n}^{-}(x) + Q_{m,n}^{+}(x) \right] \right\}
	\end{align*}
	for all $ x > 0 $ where the functions $ P_{m,n}^{\pm}(\cdot) $ and $ Q_{m,n}^{\pm}(\cdot) $ are defined as in Lemma \ref{ExpliziteBerechnungII} and 		Lemma \ref{ExpliziteBerechnungIII}, respectively.
\end{theorem}

\begin{proof}
	Under the assumptions of this theorem and with the index shift defined by $ m := \ell - j - 1 $ it is straightforward to prove the formula in Theorem 	\ref{Satz ueber die explizite Darstellung der Funktionen k}.
\end{proof}

\begin{corollary} \label{Die Funktionen k sind reell analytisch fuer positive x}
	Let $ \ell $ be in $ \mathbb{N} $. Then the function $ k^{(\ell)} $ is real analytic on $ (0,\infty) $.
\end{corollary}
  
  \section*{Acknowledgement}
		The present paper is based on the author's Master's thesis "\"Uber eine diskrete Familie von Integraloperatoren". The author would like to thank his 		 supervisor, Vadim Kostrykin, for long and fruitful discussions about Hankel operators and methods of diagonalizing them. Furthermore, the author 				would like to thank Dmitri Yafaev for a useful question, Manfred Lehn for the simplification of the proof of (2) and (3) in Lemma 											\ref{NuetzlicheSummen}, and Hans Heinrich Leithoff for reading the manuscript.


\begin{thebibliography}{5pt}
		\bibitem{Erdelyi} A.\,Erd\'{e}lyi, W.\,Magnus, F.\,Oberhettinger, F.\,G.\,Tricomi, \textsl{Tables of integral transforms, Volume 2}, McGraw-Hill, 				1954
		\bibitem{Kostrykin} V.\,Kostrykin, K.\,A.\,Makarov, \textsl{On Krein's example}, Proc. Amer. Math. Soc. \textbf{136} (2008), 2067-2071
		\bibitem{Krein} M.\,G.\,Krein, \textsl{On the trace formula in perturbation theory}, Mat. Sbornik N.S. \textbf{33(75)} (1953), 597-626 (Russian)
		\bibitem{Oberhettinger} F.\,Oberhettinger, \textsl{Tabellen zur Fourier Transformation}, Springer, 1957
		\bibitem{Power} S.\,C. Power, \textsl{Hankel operators on Hilbert space}, Pitman, Boston, MA, 1982
		\bibitem{Rosenblum} M.\,Rosenblum, \textsl{On the Hilbert matrix, II}, Proc. Amer. Math. Soc. \textbf{9} (1958), 581-585
		\bibitem{Yafaev} D.\,R.\,Yafaev, \textsl{A commutator method for the diagonalization of Hankel operators}, Funct. Anal. Appl. \textbf{44} (2010), 				no. 4, 295-306
	\end{thebibliography}
\end{document}